\documentclass[10pt]{amsart}
\usepackage{graphicx,amsfonts,amsthm,amsmath,amssymb,latexsym,color,epstopdf,xcolor,wrapfig,epsfig, amscd,mathrsfs,hyperref}
\graphicspath{ {./figures copy/} }
\usepackage[all]{xy}
\hypersetup{colorlinks,linkcolor={blue},citecolor={red}}
\usepackage{stmaryrd}
\usepackage{soul}
\usepackage[T1]{fontenc}
\usepackage{stmaryrd}
\usepackage{verbatim}
\usepackage{enumerate}
\newcommand{\eps}{\varepsilon}

\newcommand{\F}{\mathscr{F}}

\newcommand{\la}{\lambda}

\newcommand{\g}{\gamma}
\newcommand{\D}{\Delta}
\newcommand{\bbD}{\mathbb{D}}
\newcommand{\G}{\Gamma}

\newcommand{\diam}{\mathrm{diam}}
\newcommand{\dist}{\mathrm{dist}}
\newcommand{\m}{\mathrm{mod}}

\renewcommand{\d}{\delta}
\renewcommand{\d}{\delta}
\renewcommand{\O}{\Omega}

\def\XXint#1#2#3{{\setbox0=\hbox{$#1{#2#3}{\int}$}
\vcenter{\hbox{$#2#3$}}\kern-.5\wd0}}

\renewcommand{\d}{\delta}

\numberwithin{figure}{section}
\numberwithin{equation}{section}

\input epsf

\numberwithin{equation}{section}
\theoremstyle{definition}
\theoremstyle{plain}
\newtheorem{definition}{Definition}[section]
\newtheorem{remark}[definition]{Remark}
\newtheorem{example}[definition]{Example}
\newtheorem{proposition}[definition]{Proposition}
\newtheorem{question}[definition]{Question}
\newtheorem{theorem}[definition]{Theorem}
\newtheorem{corollary}[definition]{Corollary}
\newtheorem{lemma}[definition]{Lemma}

\author{Xinlong Dong and Hrant Hakobyan}

\title{Divergent geodesics in the Universal Teichm\"uller space}
\address{XD: Department of Mathematics, Kingsborough CC - CUNY. Brooklyn, NY 11235}
\email{xinlong.dong@kbcc.cuny.edu}
\address{HH: Department of Mathematics, Kansas State University. Manhattan, KS 66502}
\email{hakobyan@math.ksu.edu}
\begin{document}


\date{\today}

\maketitle

\begin{abstract}
Thurston boundary of the universal Teichm\"uller space $T(\mathbb{D})$ is the space  $PML_{bdd}(\mathbb{D})$  of projective
bounded measured laminations of $\mathbb{D}$.
A geodesic ray in $T(\mathbb{D})$ is of generalized Teichm\"uller type if it shrinks the vertical foliation of a holomorphic quadratic differential.
We provide the first examples of generalized Teichm\"uller rays which diverge near Thurston boundary $PLM_{bdd}(\mathbb{D})$. Moreover,  for every $k\geq 1$ we construct examples of rays with limit sets homeomorphic to $k$-dimensional cubes.  For the latter result we utilize the classical Kronecker approximation theorem from number theory which states that if $\theta_1,\ldots,\theta_k$ are rationally independent reals then the sequence $(\{\theta_1 n\},\ldots,\{\theta_k n\})$ is dense in the $k$-torus $\mathbb{T}^k$.
\end{abstract}

\setcounter{tocdepth}{1}
\section{Introduction}

In this paper we continue the study of asymptotic behavor of geodesics in the universal Teichm\"uller space $T(\mathbb{D})$ started in \cite{HakSar-vertical,HakSar-limits,HakSar-visual}.  In these works the main goal was to exhibit large families of geodesics which converge at Thurston's boundary $\partial_{\infty} T(\mathbb{D})$. The latter can be identified with the space of projective classes of bounded geodesic measured laminations of the hyperbolic plane $PML_{bdd}(\mathbb{D})$, cf. \cite{BonahonSaric}.  In this paper we construct the first examples of divergent geodesics in $T(\mathbb{D})$.  Moreover,  we show that the limit set of a divergent geodesic  in $PML_{bdd}(\mathbb{D})$ can be homeomorphic to a $k$-dimensional cube $[0,1]^k$, for every $k\geq 1$.  

For finite dimensional Teichm\"uller spaces the first examples of divergent Teichm\"uller geodesics were given by Lenzhen  in \cite{Lenzhen}.  Limit sets of Teichm\"uller and Weil-Petersson geodesics in the context of finite dimensional Teichm\"uller spaces were actively investigated recently,  see for instance  \cite{BLMR-wp,BLMR-wp-non-minimal, BLMR-t2, CMW, LLR,LMR}.  Thus the present paper can be thought of as a first step in obtaining analogous results for $T(\mathbb{D})$.  The techniques used in this paper are very different from the finite dimensional ones and are more analytical in nature.  Similar to \cite{HakSar-visual} we rely heavily on careful estimates of classical modulus for degenerating families of curves in domains with chimneys.  In \cite{HakSar-visual} it was essentially shown that  if $D$ is a domain with chimneys  s.t. the blowups of $D$ near boundary points converge to a half plane,  a  complement of a quadrant, or a complement of a half line then the geodesic in $T(\mathbb{D})$ corresponding to $D$ was convergent. The idea behind constructing a divergent geodesic then is to consider domains so that the blowups near a boundary point do not converge (in Hausdorff topology).  Specifically, we consider domains $D$ with infinitely many chimneys accumulating to $\{0\}\times(0,\infty)$ so that on some scales $D$ ``looks like a half-plane'' and on other scales it ``looks like a quadrant''.  Interestingly, such a construction can yield both convergent as well as divergent geodesics.  Theorems   \ref{thm:main-intro} and  \ref{thm:main} characterize which of the two possibilities occurs depending on the widths and the relative positions of the chimneys.  The rest of this introduction describes our results in more detail.

\subsection{Universal Teichm\"uller space and Thurston boundary}
Let $\mathbb{D}$ and $\mathbb{S}^1$ denote the unit disk and the unit circle in the complex plane $\mathbb{C}$,  respectively.  The \emph{universal Teichm\"uller space}, denoted by $T(\mathbb{D})$, is defined as the collection of all quasisymmetric (or qs) self-maps of $\mathbb{S}^1$, which fix $1,i$ and $-1$.  Universal Teichm\"uller space may be equipped with the Teichm\"uller metric $d_T(\cdot,\cdot)$ as follows.  Given $f,g\in T(\mathbb{D})$, the \emph{Teichm\"uller distanse} between $f$ and $g$  is defined by $d_T(f,g) = \frac{1}{2} \inf_h \log  K_{h},$
where,  the infimum is over all the quasiconformal (or qc) extensions $h:\bbD\to\bbD$ of $g\circ f^{-1}:\mathbb{S}^1\to\mathbb{S}^1$, and $K_{{h}}$ denotes the maximal dilatation of the qc map ${h}$, see Section \ref{section:boundary} for the definitions of these terms.  

Thurston boundary of $T(\mathbb{D})$, denoted by $\partial_{\infty}T(\mathbb{D})$ was defined and studied in \cite{Saric-currents,BonahonSaric}. In particular, it was shown in \cite{BonahonSaric} that $\partial_{\infty}T(\mathbb{D})$ can be identified with the space of projective bounded measured laminations of the unit disk  $\bbD$, denoted by $PML_{bdd}(\bbD)$.

In \cite{HakSar-vertical,HakSar-limits,HakSar-visual}  \v{S}ari\'{c} and the second named author studied the asymptotic behavior of geodesic rays in $T(\bbD)$.  Specifically,  suppose $\varphi$ is a holomorphic quadratic differential $\varphi$ on $\bbD$. If for every $t\in[0,1)$ the Beltrami differential $\mu_{\varphi}(t)=t|\varphi|/\varphi$ is extremal (see Section \ref{section:rays} for the definition of extremality) then the corresponding path 
\begin{align}\label{def:gen-teich-ray}
t\mapsto T_{\varphi}(t):=\left[\mu_{\varphi}(t)\right]
\end{align}
in $T(\bbD)$ is a geodesic ray, which will be called a \emph{generalized Teichm\"uller geodesic ray corresponding to $\varphi$.}  If $\varphi$ is an integrable holomorphic quadratic differential then $\mu_{\varphi}(t)$ is uniquely extremal by \cite{Str64} and hence $T_{\varphi}$ is a geodesic ray.  We call such a ray simply a \emph{Teichm\"uller geodesic ray corresponding to $\varphi$.}

In \cite{HakSar-limits} it was proved that \emph{every} Teichm\"uller geodesic ray in $T(\bbD)$ converges (in weak $^*$ topology) to a unique point in $\partial_{\infty}T(\mathbb{D})=PML_{bdd}(\bbD)$, and that distinct Teichm\"uller geodesic rays converge to distinct points in $PML_{bdd}(\bbD)$.   In particular, there is an open and dense set of geodesics in $T(\mathbb{D})$ which have distinct limits at Thurston boundary.  In \cite{HakSar-visual} if was shown that for a large class of \emph{generalized} Teichm\"uller geodesic rays the convergence to unique points of $PML_{bdd}(\bbD)$ still holds, however, distinct geodesics can converge to the same point at infinity.  In view of the above it is natural to ask if there are divergent geodesics in $T(\mathbb{D})$, and if so what are the possible sets in $PML_{bdd}(\mathbb{D})$ which can occur as limit sets of such divergent geodesics.

\subsection{Main results}

To describe our results more precisely we recall the construction of a natural class of generalized Teichm\"uller rays.  Suppose we are given a simply connected domain $D\subsetneq\mathbb{C}$ with a chimney.  This means that there is a non-trivial finite interval $(a,b)\subset\mathbb{R}$ such that $(a,b)\times(0,\infty)\subset D$ and $(\{a\}\cup\{b\})\times (t,\infty)\subset\partial D$ for some $t>0$.   Let  $\phi_{D}:\mathbb{D}\to D$ be a conformal map  and consider the holomorphic quadratic differential $\varphi_D=dw^2$, where $w(z)=\phi_D(z)$, $z\in\mathbb{D}$.  It turns out,   see e.g.  \cite{HakSar-visual}, that in this case  Beltrami differential $t|\varphi_D|/\varphi_D$ is extremal,  and the path $t\mapsto T_{\varphi_D}(t)=\left[t\frac{|\varphi_D|}{\varphi_D}\right]$ is a generalized Teichm\"uller ray.  Geometrically $t{|\varphi_D|}/{\varphi_D}$ can be thought 
of as the Beltrami coefficient corresponding to the vertical compression map $(x,y)\mapsto \left(x,\frac{1-t}{1+t} y\right)$ of $D$, which degenerates as $t\to1$, see Section \ref{section:gen-teich-rays-domains}. \v{S}ari\'{c} and the second author considered rays $\{T_{\varphi_D}(t)\}$ corresponding to domains with finitely many chimneys or infinitely many chimneys $(a_i,b_j) \times (0,\infty)$ so that the sequences $\{a_n\}$ and $\{b_n\}$ do not have accumulation points in $\mathbb{R}$. They showed that for such domains the corresponding generalized Teichm\"uller rays have limits in Thurston's boundary $\partial_{\infty}T(\mathbb{D})$, see \cite[Thm. 6.3]{HakSar-visual}.  

In this paper we consider domains with accumulating chimneys. 
Specifically,  let $\{x_n\}_{n=0}^{\infty}$ be a strictly decreasing sequences of positive numbers so that $x_n\to0$ as $n\to\infty$.  Letting $Q_4=\{ z: \Re z >0,  \Im z <0 \} $ denote the forth quadrant in $\mathbb{C}$,  we define the domain $W=W(\{x_n\})\subset\mathbb{C}$  as follows:
\begin{align}\label{domain-intro}
W= Q_4\cup \bigcup_{n=0}^{\infty} ((x_{2n+1},x_{2n}) \times \mathbb{R}).
\end{align}

In Theorem \ref{thm:main} we give a complete description of the asymptotic behavior
of the generalized Teichm\"uller ray  corresponding to $W$,  provided the sequence $x_n$ converges to $0$ fast enough.  In particular, we show that the ray  $t\mapsto T_{\varphi_{W}}(t)$ either converges or it diverges and its limit set $\Lambda\subset PML_{bdd}(\mathbb{D})$ is homeomorphic to $[0,1]$. To state our results more precisely we introduce some notation.

Given a geodesic $\g$ in the unit disk we denote by  $\delta_{\g}$ the Dirac mass at $\g$.  For every chimney $C_n=(x_{2n+1},x_{2n})\times \mathbb{R}$ we consider the  point $z_n\in\partial{\mathbb{D}}$,  which corresponds to points in $W$  which are ``in $C_n$ and are near $\infty$'', i.e., if $z\to z_n$ then $\varphi_W(z)\in C_n$ and $\Im(\varphi_W(z))\to \infty$.  Let  $\g_n^+$ and $\g_n^-$ be the  geodesics in $\mathbb{D}$ connecting the preimages of $x_{2n}$ and $x_{2n+1}$ to $z_n$, respectively.  Let  $\lambda_W$ be the geodesic measured lamination supported on these geodesics and giving each a unit mass,  i.e.,  $\lambda_W=\sum_{n=0}^{\infty} (\delta_{\g_n^+} + \delta_{\g_n^-})$. We also denote by $g$ the geodesic connecting $-1$ to $-i$. The following result shows that the limit set of $t\mapsto T_{\varphi_D}(t)$ is either a point or an interval.

\begin{theorem}\label{thm:main-intro}
Let $W$ be a domain defined as in (\ref{domain-intro}) s.t. $\frac{x_{n+1}}{x_n} \underset{n\to\infty}{\longrightarrow} 0$.
Then the limit set $\Lambda\subset PLM_{bdd}(\mathbb{D})$ of the generalized Teichm\"uller ray $t\mapsto T_{\varphi_{W}}(t)$ can be described as follows:
\begin{align}\label{limitset}
\Lambda = \{\llbracket s \delta_{g} + (2/3)\lambda_W \rrbracket : s\in[m,M]\},
\end{align}
where
\begin{align}\label{m&M:1-intro}
\begin{split}
m&=1+\liminf_{n\to\infty}\sum_{i=0}^{2n+1}  \frac{(-1)^{i} \log x_i}{\log x_{2n+1}}, \quad  M=1+\limsup_{n\to\infty} \sum_{i=0}^{2n+1} \frac{(-1)^{i} \log x_i}{\log x_{2n}}.
\end{split}
\end{align}
In particular, $t \mapsto T_{\varphi_{W}}(t)$ diverges at $\partial_{\infty}T(\mathbb{D})$ if and only if $m<M$.
\end{theorem}

Theorem \ref{thm:main-intro} is an easy consequence of Theorem \ref{thm:main} and is proved after the statement of the latter.  Applying Theorem \ref{thm:main-intro} to  $x_n\asymp {1}/{n!}$ yields an example of a  convergent  geodesic  ray with the limit point $\llbracket \frac{3}{2}\delta_g+\frac{2}{3} \lambda_W\rrbracket\in PML_{bdd}(\mathbb{D})$, see Examples \ref{example:factorial} and \ref{example:risingfactorial}. 

To obtain divergent geodesics we need to consider sequences $x_n$ which converge to $0$ faster than $1/n!$.  For instance, we have the following corollary of Theorem \ref{thm:main-intro},  see Example \ref{example:p}.

\begin{corollary}
Fix $p>1$.  Let $x_0=1,  x_1=a$ and $x_n=x_{n-1}^p$ for $n>1$ (equivalently $x_{n}=a^{p^{n-1}}$).  Then $t\mapsto T_{\varphi_{W}}(t)$ is a generalized Teichm\"uller geodesic ray which diverges at Thurston boundary with a limit set given by (\ref{limitset}),  where $m=1+\frac{1}{p+1},$  and $M=1+\frac{p}{p+1}$.
\end{corollary}

\subsubsection{Moduli of famililes of curves} 
To prove Theorems \ref{thm:main-intro} and \ref{thm:main} we restate them in terms of the limiting behavior of moduli of families of curves.  

Let $w=\phi(z)$ be the conformal map of the unit disk $\mathbb{D}$ onto $W$ such that $\phi^{-1}$ takes $w=0$ to $z=-1$,  $w=1$ to $z=1$, and as $z$ approaches $-1$ we have  that $\phi(z)$ tends to $\infty$ while staying in the fourth quadrant of $\mathbb{C}$ , see Figure \ref{figure:multi-dim}. 

Given open disjoint arcs $I,J$ on the unit circle $\partial\mathbb{D}$,  let  $\G_{I,J} = \G(I,J;\mathbb{D})$ be the curve family connecting $I$ to $J$ in $\mathbb{D}$.  For every $\eps>0$  we denote by $V_{\eps}$ the vertical compession map $(x,y)\mapsto(x,\eps y)$. Observe that if  $\eps=\eps(t)=\frac{1-t}{1+t}$, then $\eps\to0$ as $t\to1$ , that is, as the corresponding  point $[t|\phi_W|/\phi_W]\in T(\mathbb{D})$ leaves every compact subset of $T(\mathbb{D})$.

Let $\G^{\eps}_{I,J} =V_{\eps}(\phi(\G_{I,J}))$.  Using the connection between the Liouville measure and moduli of curve families,  see  Lemma \ref{lem:mod_liouville_measure} or \cite[Lemma 3.3]{HakSar-vertical}, understanding the behavior of the geodesic $t\mapsto T_{\phi_W}(t)$ is equivalent to understanding the behavior of $\m \G^{\eps}_{I,J}$  for different choices of pairs of arcs $I,J\in\mathbb{S}^1$,  as $\eps\to0$.  Namely,  to show that a point $\llbracket\lambda\rrbracket\in PML_{bdd}(\mathbb{D})$ is a limit point of the generalized Teichm\"uller geodesic $t\mapsto T_{\phi_W}(t)$ it is sufficient to show that 
there is a sequence $\eps_n\to0$ so that an appropriate rescaling of the sequence  $\m \G^{\eps_n}_{I,J}$ converges to $\lambda(I\times J)$.  

It follows essentially from  \cite{HakSar-visual} that if the intersection of the box $I\times J$ and geodesic lamination $\mathfrak{L}= \{g\}\cup\bigcup_{n=0}^{\infty}\{\g_n^+,\g_n^-\}$ is empty then $$\lim_{\eps\to0}\frac{\m \G^{\eps}_{I,J}}{ \frac{1}{\pi}\log \frac{1}{\eps}}=0.$$ On the other hand, if $(I\times J)\cap\mathfrak{L}$ contains a single geodesic from $\mathfrak{L}$ that is not $g$ then $$\lim_{\eps\to0}\frac{\m \G^{\eps}_{I,J}}{\log \frac{1}{\eps}}=\frac{2}{3}.$$

If $(I\times J)\cap\mathfrak{L} = g$ we prove the following.

\begin{lemma}\label{lemma:limsup&liminf-intro}
{Suppose $W$ satisfies the conditions of Theorems \ref{thm:main-intro} (or Theorem \ref{thm:main}).  If $(I\times J) \cap \mathfrak{L} =g$  then 
\begin{align}\label{eqn:modlimits}
\liminf_{\eps\to0}\frac{\m \G^{\eps}_{I,J}}{\frac{1}{\pi} \log \frac{1}{\eps}} = m, \quad 
\limsup_{\eps\to0}\frac{\m \G^{\eps}_{I,J}}{\frac{1}{\pi} \log \frac{1}{\eps}} = M.
\end{align}}
\end{lemma}

Lemma  \ref{lemma:limsup&liminf-intro} is a key result used in the proof of Theorems \ref{thm:main-intro} and \ref{thm:main}.  It follows from Proposition \ref{lemma:independence} and  Lemma \ref{lemma:limsup&liminf}.

 \subsubsection{Divergent geodesics with higher dimensional limit sets}
 

Using Theorem \ref{thm:main-intro}  we construct for every $k\in\mathbb{N}$ a domain $\mathcal{W}_k$ such that the limit set corresponding to the ray $t\to T_{\varphi_{\mathcal{W}_k}}(t)$ is homeomorphic to $[0,1]^k$.  For instance,  for $k=2$ we have the following result, which follows from Theorem \ref{thm:divergent-square}.

\begin{figure}\label{figure:multi-dim}
	\centerline{
		\includegraphics[width=0.8\textwidth]{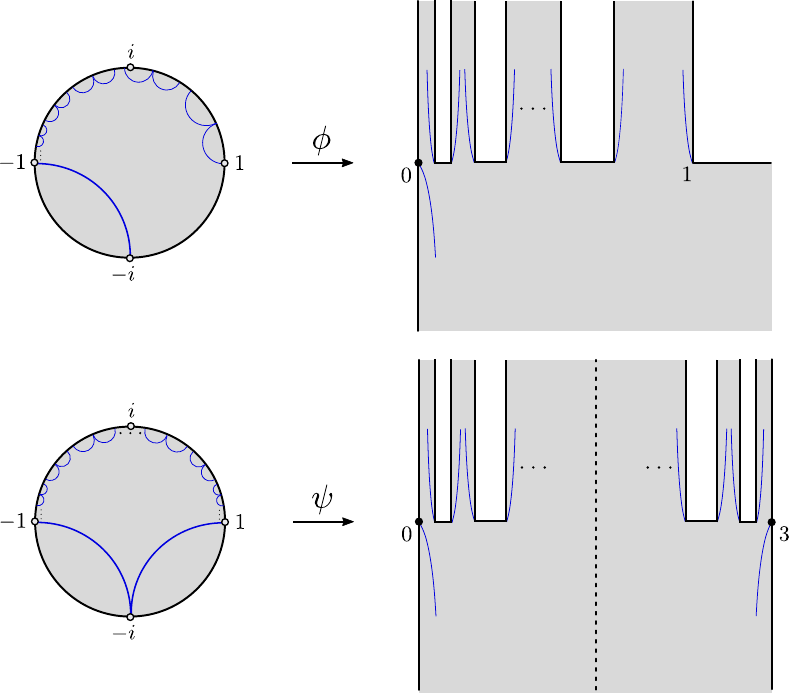}
	}
	\caption{\small{Supports of the limiting geodesic measured laminations corresponding to the divergent geodesic rays constructed in Theorems \ref{thm:main-intro} and \ref{thm:divergent-square-intro}}.}
\end{figure}

\begin{theorem} \label{thm:divergent-square-intro}
Fix $a\in(0,1)$, and $p,q>1$ so that $\log p/ \log q$ is irrational.   Let $x_{n}=a^{p^{n-1}}$, $x_n' = 3-a^{q^{n-1}}$, and $C_n,  C_n'$ be chimneys over the intervals $(x_{2n+1},x_{2n})$ and 
$(x'_{2n},x_{2n+1}')$, respectively.  Then the generalized Teichm\"uller ray $t\mapsto T_{\varphi_{\mathcal{W}_{p,q}}}(t)$ corresponding to the domain  
$$W_{p,q}=\{z \,:\, 0<\Re (z)<3, \,  \Im (z)>0\} \cup \bigcup_{n=1}^{\infty} C_n\cup C'_n$$ diverges at $\partial_{\infty}T(\mathbb{D})$ and its limit set $\Lambda\subset PML_{bdd}(\mathbb{D})$ in Thurston boundary is homeomorphic to $[0,1]^2$.  
\end{theorem}

A key result used in the proof of Theorem \ref{thm:divergent-square} is the dynamical fact that every orbit $\{T_{\theta}^{\circ n}(x)\}_{n=1}^{\infty}$ of the irrational rotation $T_{\theta} : x\mapsto \{ \theta +x\}$  is dense in the circle $\mathbb{R}/\mathbb{Z}$, where $\theta=\log p / \log q$ and $\{ y \}=y-\lfloor y \rfloor$ denotes the fractional part of a real number $y$. 

To prove a higher dimensional version of Theorem \ref{thm:divergent-square-intro},  see Theorem \ref{thm:kd},  we construct domains with $k$ families of chimneys accumulating to distinct half-lines corresponding to sequences of the form $\{a^{p_i^n}\}_{n=1}^{\infty}, i\in\{1,\ldots,k\},$ such that the numbers $\log p_1, \ldots, \log p_k$ are rationally independent.  We use a well known approximation theorem of Kronecker, which generalizes the above mentioned fact about the density of the orbits of an irrational rotation of $\mathbb{S}^1$ to higher dimensional tori $\mathbb{T}^k=\mathbb{S}^1 \times \ldots \times \mathbb{S}^1$.


In Theorems  \ref{thm:main-intro} and \ref{thm:divergent-square-intro} if $\llbracket\lambda_1\rrbracket,\llbracket \lambda_2 \rrbracket\in\Lambda$ then the supports of $\lambda_1$ and $\lambda_2$ are the same.  This raises the following question.
\begin{question}
Is there a limit set $\Lambda\subset PML_{bdd}(\mathbb{D})$ and geodesic measured laminations $\lambda_1$ and $\lambda_2$,  so that $\llbracket\lambda_1\rrbracket,\llbracket \lambda_2 \rrbracket\in\Lambda$ and the supports of $\lambda_1$ and $\lambda_2$ are distinct.
\end{question}
 It would be interesting to undestand the possible topology of limit sets of geodesic rays in $T(\mathbb{D})$.  For instance,  even the following basic question seems to be open.  
 \begin{question}
 Is there a Teichm\"uller geodesic in $T(\mathbb{D})$ so that the limit set $\Lambda\subset PML_{bdd}(\mathbb{D})$ has nontrivial topology (e.g.  $\pi_1(\Lambda) \neq \{0\}$)?
\end{question}

The rest of this paper is organized as follow. In Section \ref{section:background} we provide the necessary definitions, notation, and some auxiliary results. In Section \ref{section:results} we state Theorem \ref{thm:main}, prove Theorem \ref{thm:main-intro} and provide several explicit examples of divergent and convergent geodesics in $T(\mathbb{D})$.  Sections \ref{section:independence} and \ref{section:mod-bounds} are the technical core of the paper,  which are devoted to the proof of  Theorem  \ref{thm:main}.  In Section \ref{section:higher-dim} we prove the existence of generalized Teichm\"uller geodesics with higher dimensional limit sets.


\section{Background}\label{section:background} 
\subsection{{{Visual boundary of the universal Teichm\"uller space}}}\label{section:boundary}
If $E$ is a subset of $\mathbb{S}^1$ we will denote by $|E|$ the arclength (or Hausdorff $1$-measure) of $E$. 

A homeomorphism $h:\mathbb{S}^1\to\mathbb{S}^1$ is said to be \emph{quasisymmetric} if there is a constant $1\leq M <\infty$ such that 
\begin{align}
\frac{1}{M}\leq \frac{|h(I)|}{|h(J)|}\leq M,
\end{align}
for all circular arcs $I, J$ in $\mathbb{S}^1$ with a common boundary point and disjoint interiors such that $|I|=|J|$. A homeomorphism is quasisymmetric if and only if it extends to a quasiconformal map of the unit disk.

The \emph{universal Teichm\"uller space}, denoted by $T(\mathbb{D})$, consists of all quasisymmetric $h: \mathbb{S}^1 \to \mathbb{S}^1$, which fix $1,i$ and $-1$. 
The \emph{Teichm\"uller distanse} between elements $f$ and $g$ in $T(\mathbb{D})$ is defined as follows.:
\begin{align*}
d_T(f,g) = \frac{1}{2} \inf\log  K_{g\circ f^{-1}},
\end{align*}
where, given a quasisymmetric mapping $h$ of $\mathbb{S}^1$, we denote 
\begin{align*}
K_h = \inf \{ K_{\tilde{h}} \, | \, \tilde{h}:\mathbb{D}\to\mathbb{D} \mbox{ is a quasiconformal extension of } h\},
\end{align*}
with $K_{\tilde{h}}$ denoting the maximal dilatation of the quasiconformal map $\tilde{h}$ extending $h$ to $\mathbb{{D}}$. We refer to \cite{Ahlfors:QClectures} for background on quasisymmetric and quasiconformal maps and, in particular,  for the definitions of the maximal dilatation  of a quasiconformal map and quasiconformal extensions.

The universal Teichm\"uller space $T(\mathbb{D})$ may also be defined as the set of equivalence classes of Beltrami coefficients $[\mu]\in B_1/\sim$, where $B_1$ is the unit ball in $L_{\infty}(\mathbb{D})$, and $\mu\sim\nu$ whenever the corresponding quasiconformal mappings $f^{\mu},f^{\nu}:\mathbb{D}\to\mathbb{D}$  coincide on the boundary circle, i.e. $f^{\mu}|_{\mathbb{S}^1} = f^{\nu}|_{\mathbb{S}^1}$, see \cite{GL}. Given two Beltrami coefficients  $\mu_0$ and $\nu_0$ the Teichm\"uller distance between $[\mu_0]$ and $[\nu_0]$ in $T(\mathbb{D})$ is defined as follows
\begin{align}\label{teich-distance}
	d_T([\mu_0],[\nu_0]) = \frac{1}{2} \inf_{\substack{\mu\in[\mu_0] \\ \nu\in[\nu_0]}} \log \frac{1+\left\|\frac{\mu-\nu}{1-\bar{\mu}\nu}\right\|_{\infty}}{1-\left\|\frac{\mu-\nu}{1-\bar{\mu}\nu}\right\|_{\infty}}.
\end{align}

Given a quasisymmetric mapping $h:\mathbb{S}^1\to\mathbb{S}^1$, a quasiconformal mapping $f:\mathbb{D}\to\mathbb{D}$ continuously extending $h$ is called an  \textit{extremal quasiconformal mapping} (for its boundary values) if it has the smallest maximal dilatation $K_f$ among all such extensions of $h$ to $\mathbb{D}$.

A Beltrami coefficient $\mu$ is said to be {\it extremal} or {\it{uniquely extremal}} if
$\|\mu\|_{\infty}\leq \|\nu\|_{\infty}  \mbox{ or }  \|\mu\|_{\infty} < \|\nu\|_{\infty},$
respectively, whenever $\nu\sim\mu$ and $\nu\neq\mu$. If $\mu$ is an extremal Beltrami coefficient then
\begin{align*}
	\mu_s=\frac{(1+|\mu|)^s-(1-|\mu|)^s}{(1+|\mu|)^s+(1-|\mu|)^s} \frac{\mu}{|\mu|},
\end{align*}
is also extremal for every $s\in[0,\infty)$, see \cite{FM}.  Moreover, with the notation as above we have
\begin{align*}
	\dist_T([0],[\mu_s]) = s\cdot \dist_T([0],[\mu]).
\end{align*}
Therefore, the path $s\mapsto [\mu_s]$ is a geodesic  ray in the Teichm\"uller metric in $T(\mathbb{D})$, i.e. it is a geodesic path starting at $[0]$, passing through $[\mu]$ and leaving every compact set in $T(\mathbb{D})$ as $s\to\infty$. The collection of all geodesic rays in $T(\mathbb{D})$ is the \emph{visual boundary} of the universal Teichm\"uller space, see \cite{HakSar-visual}.

\subsection{{Thurston boundary of $T(\mathbb{D})$}}

In \cite{Bon1} Thurston's boundary of the Teichm\"uller space $T(S)$ of a closed surface $S$ was defined by embedding $T(S)$ into the space of geodesic currents on $S$.  In  \cite{Saric-currents, Sa3, BonahonSaric} the definition for a Thurston type boundary of Techm\"uller spaces was given for all Riemann surfaces,  and in particular for $T(\mathbb D)$.  We briefly recall some of the relevant notation, see  \cite{HakSar-visual}.

The space $G(\mathbb{D})$ of oriented geodesics on $\mathbb{D}$ can be identified with $\mathbb{S}^1\times \mathbb{S}^1\setminus diag$.  Observe that a quasisymmetric mapping $f: \mathbb{S}^1 \to \mathbb{S}^1$ induces a self-map of the space of geodesics $G(\mathbb{D})$ by mapping a hyperbolic geodesic $\g_{(x,y)}$ to $\g_{(f(x),f(y))}$, for every pair of distinct point $x,y\in\mathbb{S}^1$. We denote by $\tilde{f}$ the mapping induced by $f$ on the space of geodesics.

A {\it geodesic current} is a Radon measure on $G(\mathbb{D})$.
The {\it Liouville current} $\mathcal{L}$ is a geodesic current such that for every Borel set $A\subset \mathbb{S}^1\times \mathbb{S}^1 \setminus diag$ we have
$$
\mathcal{L}(A)=\int_A \frac{|dx| |dy|}{|x-y|^2}.
$$
An easy calculation shows that for a  {\it box of geodesics} $A=[a,b]\times [c,d]$ we have 
$$
\mathcal{L}([a,b]\times [c,d])=\log\frac{(a-c)(b-d)}{(a-d)(b-c)}.
$$

The universal Teichm\"uller space $T(\mathbb{D})$ can be embedded into the space $\mathcal{C}(\mathbb{D})$ of geodesic currents via the \emph{Liouville embedding} by setting 
$$\mathscr{L}(f)=(\tilde{f})^*(\mathcal{L})\in \mathcal{C}(\mathbb{D}),$$
where the right hand side denotes the pullback of the Liouville measure by $\tilde{f}$. Moreover, whenever $f$ is a quasisymmetric then $\mathscr{L}(f)$ is a \emph{bounded geodesic current} in the sense that $\sup_A \mathscr{L}(f)(A)<\infty$, where the supremum is over all geodesic boxes $A$ such that $\mathcal L(A)=\log 2$ (thus $A$ is a Mobius image of the box $[1,i]\times[-1,-i]$).



Liouville embedding $\mathscr{L}$ is in fact a homeomorphism
of $T(\mathbb{D})$ onto its image in $\mathcal{C}(\mathbb{D})$ equipped with the uniform weak* topology, see \cite{MiSar, Sa3}.  As a set \emph{Thurston boundary of $T(\mathbb{D})$}, which is denoted by $\partial_{\infty}T(\mathbb{D})$, 
is defined as the collection of asymptotic rays to $\mathscr L(T(\mathbb{D}))$ in $\mathcal{C}(\mathbb{D})$.  Equivalently, $\partial_{\infty}T(\mathbb{D})$ can be
identified with the space of projective classes of bounded measured laminations on $\mathbb{D}$, or  $PML_{bdd}(\mathbb{D})$, see
\cite{Saric-currents, Sa3}. 


\subsection{Generalized Teichm\"uller rays in $T(\mathbb{D})$. }\label{section:rays}
In this paper we continue the study of behavior of geodesic rays in $T(\mathbb{D})$ as they approach infinity, i.e. leave every compact subset of $T(\mathbb{D})$ started in \cite{HakSar-vertical,HakSar-limits,HakSar-visual}.  Specifically,  let $\varphi$ be a holomorphic quadratic differential on $\mathbb{D}$.  Suppose that $\mu_{\varphi}(t):=t|\varphi|/\varphi$ is an extremal Beltrami differential for some (or equivalently all) $t\in(0,1)$.  By (\ref{teich-distance}) we have 
\begin{align*}
d_{T}([0],[\mu_{\varphi}(t)]) = \frac{1}{2} \log \frac{1+\|\mu_{\varphi}(t)\|_{\infty}}{1-\|\mu_{\varphi}(t)\|_{\infty}} = \frac{1}{2} \log \frac{1+t}{1-t}.
\end{align*}
Therefore, if $0<s<t<1$ then by (\ref{teich-distance}) again we have
$$d_T([\mu_{\varphi}(s)],[\mu_{\varphi}(t)])=d_T([0],[\mu_{\varphi}(t)])-d_T([0],[\mu_{\varphi}(s)]).$$
Hence the path 
\begin{align}\label{ray}
t\mapsto [t{|\varphi |}/{\varphi}]
\end{align}
 is a geodesic ray in the Teichm\"uller metric in $T(\mathbb{D})$.  We say that the path (\ref{ray}) is a  \emph{generalized Teichm\"uller ray} if the Beltrami coefficient $t\frac{|\varphi |}{\varphi}$ is extremal for some (or equivalently all) $t\in(0,1)$.





\subsection{Generalized Teichm\"uller rays and vertical compression}\label{section:gen-teich-rays-domains}
A natural class of generalized Teichm\"uller rays can be obtained as follows.  Given a simply connected domain $D\subsetneq\mathbb{C}$ and a conformal map $\phi_{D}:\mathbb{D}\to D$ consider the holomorphic quadratic differential $\varphi_D=dw^2$, where $w(z)=\phi_D(z)$, $z\in\mathbb{D}$. If the Beltrami differential $t|\varphi_D|/\varphi_D$ happens to be extremal for some $t\in(0,1)$  then the generalized Teichm\"uller ray $T_{\varphi_D}(t)=[\mu_{\varphi_{D}}(t)]$ for $\varphi=\varphi_D$ can alternatively be described as follows, see \cite{HakSar-visual}  

For every $0<\eps\leq 1$, let $V_{\eps}$ be the vertical compression map of the plane, i.e.
\begin{align*}
  V_{\eps}(x,y)=(x,\eps y),
\end{align*}
and let $\phi_{D,\eps}$ be a conformal map of the disc $\mathbb{D}$ onto $V_{\eps}(D)$ and let $\phi_D:=\phi_{D,1}:\mathbb{D}\to D.$
The mapping
$$\psi_{D,\eps}:=\phi_{D,\eps}^{-1}\circ V_{\eps}\circ \phi_D$$ is a quasiconformal mapping of $\mathbb{D}$ onto itself with the Beltrami differential given by $\left(\frac{1-\eps}{1+\eps}\right)  \frac{|\varphi_D|}{\varphi_D}$ and the maximal dilatation  $1/\varepsilon$.
%
%
%
{If $V_{\eps}$ (respectively, $\phi_{D,\eps}^{-1}\circ V_{\eps}\circ \phi_D$) is an extremal quasiconformal map for its boundary values on $\partial{D}$ (respectively, on $\mathbb{S}^1$), then the path
$$\eps\mapsto \tau_{\varepsilon,{D}}= \left[\left(\frac{1-\eps}{1+\eps}\right)  \frac{|\varphi_D|}{\varphi_D}\right]$$ 
with $\varepsilon$ decreasing from $1$ to $0$  is a generalized Teichm\"uller ray in $T(\mathbb{D})$, and we have
\begin{align*}
  \dist([0],\tau_{\eps,D}) =\frac{1}{2}\log \eps^{-1} \underset{\eps\to0} \longrightarrow \infty.
\end{align*}}

\subsection{From Teichm\"uller rays to conformal modulus}\label{section:trays-modulus}

Suppose $D$ is a domain in the complex plane as above, and $\eps\mapsto \tau_{\eps,D}$ is the corresponding Teichm\"uller ray.  Abusing the notation slightly, we will denote by $\psi_{D,\eps}$ (or $\psi_{\eps}$), if $D$ is clear from the context) the corresponding quasisymmetric mapping of $\mathbb{S}^1$. Therefore, the Teichm\"uller ray corresponding to $D$ (or $\varphi_{D}$) can also be represented as follows:
\begin{align*}
  \eps\mapsto\psi_{\eps}=\psi_{D,\eps}.
\end{align*}

To study the asymptotic behavior of the geodesic $\eps\mapsto\psi_{\eps}$ in $T(\mathbb{D})$ near $\partial_{\infty}T(\mathbb{D})$, i.e.,  as $\eps\to0$, in \cite{HakSar-vertical,HakSar-limits,HakSar-visual}  the authors used Liouville embedding of the Teichm\"uller space and a connection between Liouville measure and the classical modulus of families and curves.  

\subsubsection{Limit sets of geodesic rays in $T(\mathbb{D})$} Let $\la$ be a bounded measured lamination of the Hyperbolic plane $\mathbb{D}$.  We will denote by $\llbracket\la\rrbracket$ the projective class of $\la$.  We say that a bounded projective measured lamination $\llbracket{\lambda}\rrbracket\in PML_{bdd}(\mathbb{D})$ is a \emph{limit point} of a geodesic ray  $\eps\mapsto\psi_{D,\eps}$ if there is a sequence of positive numbers $\{\eps_i\}_{i=1}^{\infty}$ approaching $0$ such that
\begin{align}\label{Thurston-limit}
  \llbracket{\mathscr{L}(\psi_{D,\eps_i})}\rrbracket \underset{i\to\infty}\longrightarrow \llbracket{\la}\rrbracket,
\end{align}
in the weak* topology. The \emph{limis set} of the generalized Teichm\"uller ray $\eps\mapsto\psi_{D,\eps}$ is the collection of all its limit points in $PML_{bdd}(\mathbb{D})$.

To obtain a more explicit expression of the limit points of the rays in $T(\mathbb{D})$ we recall that for a given box $[a,b]\times[c,d]\in G(\mathbb{D})$ and $\eps>0$ we have
\begin{align*}
 \mathscr{L}(\psi_{\eps})([a,b]\times[c,d]) &=  \tilde{(\psi_{\eps})}^*(\mathcal{L})([a,b]\times[c,d])\\
  &= \mathcal{L}(\psi_{\eps}([a,b]) \times \psi_{\eps}([c,d])).
\end{align*}
Denoting $x^{\eps}:=\psi_{\eps}(x)$ for every $x\in\mathbb{S}^1$ we can write the equality above as
\begin{align*}
 \mathscr{L}(\psi_{\eps})([a,b]\times[c,d]) = \mathcal{L}([a^{\eps},b^{\eps}] \times [c^{\eps},d^{\eps}]).
\end{align*}
Thus, to find the limit points of a generalized Teichm\"uller geodesic we need to find the asymptotic behavior of $\mathcal{L}([a^{\eps},b^{\eps}] \times [c^{\eps},d^{\eps}])$ as $\eps$ approaches $0$.  To do this we recall the notion of modulus of families of curves.

\subsubsection{From Liouville measure to conformal modulus. }\label{section:modulus}
Let $\G$ be a family of curves in a domain $\O\subset\mathbb{C}$.  A non-negative Borel function $\rho$ on $\O$ is called \emph{admissible for $\G$},  if $l_{\rho}(\gamma ):=\int_{\gamma}\rho (z)|dz|\geq 1$ for every  $\g\in\G$.
Conformal modulus of $\G$ then is defined by
\begin{align*}
	\m \G =  \inf_{\rho \mbox{\tiny{ is $\G$-adm}}} \iint_{\O} \rho^2 dxdy,
\end{align*}
where the infimum is over all $\G$ admissible metrics $\rho$.

Lemmas \ref{lemma:modproperties}, \ref{lemma:modrectangle} and \ref{lemma:modannulus} below, summarize some of the main properties
of the modulus, which we will use below, see \cite{GM,LV,Vaisala:lectures}.

If $\G_1$ and $\G_2$ are curve families in $\mathbb{C}$, we will say that
$\G_1$ \emph{overflows} $\G_2$ and will write $\G_1>\G_2$ if every curve
$\g_1\in \G_1$ contains some curve $\g_2\in \G_2$.

\begin{lemma} \label{lemma:modproperties}
	Let $\G_1,\G_2,\ldots$ be curve families in $\mathbb{C}$. Then
	\begin{itemize}
		\item[1.] \textsc{Monotonicity:} If $\G_1\subset\G_2$ then $\m(\G_1)\leq
		\m(\G_2)$.
		\item[2.] \textsc{Subadditivity:} $\m(\bigcup_{i=1}^{\infty} \G_i) \leq
		\sum_{i=1}^{\infty}\m(\G_i).$
		\item[3.] \textsc{Overflowing:} If $\G_1>\G_2$ then $\m \G_1 \leq \m
		\G_2$.
		\item[4.] \textsc{Conformal invariance:} If $f:\O\to\O'$ is a conformal map then $$\m f(\G) = \m \G$$ for any curve family $\G$ in $\O$, where $f(\G)$ denotes the family of curves $\{f(\g) : \g\in\G\}$ in $\Omega'$.
	\end{itemize}
\end{lemma}
The next two examples of curve families are fundamental and will be used repeatedly throughout the paper, see \cite{Ahlfors:QClectures}.
\begin{lemma}[Modulus of a rectangle] \label{lemma:modrectangle}
	Let $R=[0,l]\times[0,w]$ and $\G$ be the family of curves in $R$ connecting the vertical sides, i.e., $\{0\}\times[0,w]$ and $\{l\}\times[0,w]$. Then $\m \G = w/l$.
\end{lemma}

\begin{lemma}[Modulus of an annulus] \label{lemma:modannulus}
	Suppose $A=\{z: 0<r_1<|z|<r_2\}$. Let $\G_A$ and $\G_A'$ be the families of  curves in $A$ separating and connecting the two boundary components of $A$, respectively. Then 
	\begin{align*}
		\m \G_A  = \frac{\log \frac{r_2}{r_1}}{2\pi},\qquad 
		\m \G_A' = \frac{2\pi}{\log \frac{r_2}{r_1}}.
	\end{align*}
\end{lemma}

The following estimate for modulus will be used repeatedly below.

Given two continua $E$ and $F$ in $\mathbb{C}$ we denote
\begin{align*}
  \D(E,F) := \frac{\mathrm{dist}(E,F)}{\min\{\diam E, \diam F\}},
\end{align*}
and call $\D(E,F)$ the \textit{relative distance} between $E$ and $F$ in
$\mathbb{C}$.

\begin{lemma}[cf. \cite{HakSar-limits}]\label{lemma:mod-reldist}
For every pair of continua $E,F\subset\mathbb{C}$ we have
\begin{align}\label{modest:reldistance}
\m(E,F;\mathbb{C}) \leq \pi\left(1+\frac{1}{2\D(E,F)}\right)^2.
\end{align}
\end{lemma}

In particular  if $E_n$ and $F_n$ are such that $\D(E_n,F_n)$ is bounded away from $0$ then  $\m(E_n,F_n;\mathbb{C})$ is bounded above.

The following result,  connects Liouville measure and the moduli of curve fmilies and is central for our analysis.
\begin{lemma}[See Lemma 3.3 in  \cite{HakSar-vertical}]\label{lem:mod_liouville_measure} 
	Let $(a,b,c,d)$ be a quadruple of points on
	$\mathbb{S}^1$ in the counterclockwise order, and let $\G_{[a,b]\times [c,d]}$ be the family of curves in $\mathbb{D}$ connecting
	$[a,b]$ to $[c,d]$. Then
	\begin{align}\label{mod-Liouville}
		\m(\G_{[a,b]\times [c,d]})-\frac{1}{\pi}\mathcal{L}([a,b]\times [c,d])-\frac{2}{\pi}\log 4\to 0
	\end{align}
	as $\m(\G_{[a,b]\times [c,d]})\to\infty$, where $\mathcal{L}$ is the
	Liouville measure.
\end{lemma}

\begin{remark} Note that by equation (\ref{mod-Liouville}) we have that $\m(\G_{[a,b]\times [c,d]})\to\infty$ if and only if  $\mathcal{L}([a,b]\times [c,d])\to\infty$. Therefore it is enough to consider the asymptotic behavior of the modulus in order to find the asymptotic behavior of the Liouville measure.
\end{remark}

We next give a criterion from \cite{HakSar-visual} which will help identify which elements of $PLM_{bdd}(\mathbb{D})$ can occur as limit points at infinity of generalized Teichm\"uller rays in $T(\mathbb{D})$.

\begin{definition}
A set $\mathcal{B}$ of boxes of geodesics is said to be dense among all boxes of geodesics if for any box $[a,b]\times [c,d]\subset G(\mathbb{D})$ there exists a sequence $\{ [a_n,b_n]\times [c_n,d_n]\}_n$ in $\mathcal{B}$ such that $a_n\to a$, $b_n\to b$, $c_n\to c$ and $d_n\to d$ as $n\to\infty$.
\end{definition}

\begin{corollary}\label{corol:modformulation}
  Let $D\subset\mathbb{C}$ be a domain such that $\eps\mapsto \psi_{D,\eps}$ is a generalized Teichm\"uller ray. Suppose that there is a bounded geodesic current $\la$, a  function $m=m(\eps)\to\infty,$  as $\eps\to0,$ and a sequence $\eps_i\to0$ such that the limit
  \begin{align}\label{limit-modratio}
   \la([a,b]\times [c,d])=\lim_{i\to\infty} \frac{\m V_{\eps_i}(\phi_D(\G_{[a,b]\times [c,d]}))}{m(\eps_i)}
  \end{align}
  exists for a box of geodesics $[a,b]\times [c,d]$.  Then $\llbracket\la\rrbracket$ is a limit point of the generalized Teichm\"uller ray $\eps\mapsto \psi_{D,\eps}$ in the weak* topology if and only if (\ref{limit-modratio}) holds for a dense set $\mathcal{B}$ of boxes of geodesics in $G(\mathbb{D})$. 
\end{corollary}

We note that in \cite{HakSar-visual} the sufficiency of the condition (\ref{limit-modratio}) for a dense set of boxes was proved for the entire ray $\eps\mapsto \psi_{D,\eps}$ rather than for a sequence of points $\psi_{D,\eps_i}$,  but the same proof gives the more general result above. On the other hand, the necessity of (\ref{limit-modratio}) follows from (\ref{Thurston-limit}) by using (\ref{mod-Liouville}) and recalling the relevant definitions.

In some cases it is easier to find the limiting geodesic measured lamination using a localization result which we formulate next.

For $\eta\in(0,\pi)$ we will denote by $B_{\eta}(\g)$ the $\eta$-box containing $\g$, which may be thought of
as the $\eta$-neighborhood of $\g$. More precisely, if 
$\g =\g_{(a,b)}\in G(D)$ for some $a = e^{i\alpha}$ and $b = e^{i\beta}$ in $\mathbb{S}^1$, we let
\begin{align*}
B_{\eta}(\g)=(e^{i(\alpha-\eta)},e^{i(\alpha+\eta)})\times
(e^{i(\beta-\eta)},e^{i(\beta+\eta)}).
\end{align*}

With this notation we have the following consequence of Corollary \ref{corol:modformulation}, which is an analogue of Lemma 6.6 in \cite{HakSar-visual} and can be proved by selecting appropriately the family $\mathcal{B}$ of boxes and using the basic properties of the modulus from Lemma \ref{lemma:modproperties}. 


\begin{lemma}\label{lemma:discrete-laminations}
Let $D\subset\mathbb{C}$ be a domain such that $\eps\mapsto \psi_{D,\eps}$ is a generalized Teichm\"uller geodesic, and  $\{\eps_i\}_{i=0}^{\infty}$ is a positive sequence which converges to $0$ as  $i\to\infty$. Suppose for every $\g\in G(\mathbb{D})$ there is a positive number $\eta=\eta_{\gamma}>0$ s.t. the limit 
\begin{align}\label{weight-of-geodesic}
m(\g)=m_{\eta}(\gamma):=\lim_{i\to\infty}\frac{\m V_{\eps_i}(\phi(\G_{B_{\eta}(\g)}))}{\log\frac{1}{\eps_i}}
\end{align} 
exists for all $0<\eta<\eta_{\g}$ and is independent of $\eta$. If there is a countable collection of geodesics $\{\g_j\}_{j=1}^{\infty}$ s.t. $m(\g)>0$ if and only if $\g=\g_j$ for some $j\geq 1$, then 
\begin{align*}
\tau_{\eps_i,D}\underset{i\to\infty}{\longrightarrow} \left\llbracket \sum_{j=1}^{\infty} m(\g_j) \d_{\g_j} \right\rrbracket.
\end{align*}
\end{lemma}


\subsection{Chimney domains and generalized Teichm\"uller rays}\label{subsection:chimney}
We say that $\Omega\subset\mathbb{C}$ is a \emph{domain with a chimney} if $\Omega$ contains  a  product subset $C_{(a,b)} :=(a,b)\times (0,\infty)$ such that $\partial C_{(a,b)}\cap\{z : \Im z>\alpha\}\subset\partial \Omega$ for some $\alpha\geq 0$.

 If $\Omega$ is a domain with a chimney then by Theorem 6.1 in \cite{HakSar-visual} the vertical compression map $V_{\eps}:(x,y)\mapsto(x,\eps y)$ is an extremal quasiconformal map for its boundary values on $\partial{\Omega}$.  Therefore,  for any domain $\Omega$ with a chimney the path $\eps\mapsto\psi_{\Omega,\eps}$ defined in Section \ref{section:gen-teich-rays-domains} is a generalized Teichm\"uller geodesic.  Hence,  to conclude that $\la\in PML_{bdd}(\mathbb{D})$ is a limit point of the generalized Teichm\"uller ray $\eps\mapsto\psi_{\Omega,\eps}$ for a chimney domain $\Omega$,  by Corollary \ref{corol:modformulation} it is enough to verify that equality (\ref{limit-modratio}) holds for a dense set of boxes of geodesics. 

\section{Main result and some examples}\label{section:results}

Let $\{a_n\}$ and $\{b_n\}$ be two sequences of positive numbers so that for $0<a_0<b_0=1$,  and for  $n\geq1$ we have
$ a_{n+1}<b_n<a_n,$ and $a_n,b_n\to0$ as $n\to\infty$. 
 
For $n\geq 0$ let $I_n:=(a_n,b_n)$ and $C_n:=(a_n,b_n)\times[0,\infty)$ be the ``chimney'' over the interval $I_n$.  We define the domain $W=W(\{a_n\},\{b_n\})\subset\mathbb{C}$  as follows:
\begin{align}\label{domain}
W=\{ z: \Re z >0,  \Im z <0 \} \bigcup_{n=0}^{\infty} C_n,
\end{align}
and observe that $W$ is a simply connected  domain with a sequence of chimneys which accumulate to the part of the imaginary axis in the upper half-plane, i.e., $\{z\in\mathbb{C} : \Re{z}=0, \Im{z} \geq 0\}$. Note that $W$ is completely determined by the sequences $\{a_n\}_{n\in\mathbb{N}}$ and $\{b_n\}_{n\in\mathbb{N}}$  

Let $\phi:\mathbb{D}\to W$ be the conformal map with the following properties
\begin{align}
\begin{split}
\lim_{z\to 0}  \phi^{-1}(z) &= -1\\
\lim_{z\to1} \phi^{-1}(z) &=1\\
\lim_{\Im (z) \to -\infty} \phi^{-1}(z) &= -i.
\end{split}
\end{align}

By the discussion in Section \ref{subsection:chimney} the path $\eps\mapsto\psi_{W,\eps}$ is a generalized Teichm\"uller geodesic.  We will show that this geodesic does not converge to a unique point in $\partial_{\infty}T(\mathbb{D})$.  Moreover,  we will describe the limit set of the geodesic completely and will show that it is in fact homeomorphic to a one dimensional simplex (an interval).  To describe the limit set of $\{\psi_{W,\eps}\}$ we first introduce some notation.

We denote by $z_n$ the points on $\mathbb{S}^1$ corresponding to the chimney $C_n$ for $n\geq 0$.  More specifically,  denoting $w=\phi(z)\in W$, for $n\geq 0$ we let
\begin{align*}
z_n:=\lim_{\substack{\Im(w)\to +\infty \\ \Re(w)\in I_n} } \phi^{-1}(w).
\end{align*}
We also define
\begin{align*}
\alpha_n = \lim_{w\to a_n} \phi^{-1}(w), \quad \beta_n = \lim_{w\to b_n} \phi^{-1}(w).
\end{align*}
We will denote by $\g_{(x,y)}$ the (unoriented) hyperbolic geodesic in $\mathbb{D}$ with endpoints $x,y\in\mathbb{S}^1$.  Moreover,  for $n\geq 0$ we denote by $\g_n^{+}$ and  $\g_n^{-}$ the two geodesics starting at $z_n$ and ending at $\alpha_n$ and $\beta_n$, respectively:
\begin{align*}
\g_n^- := \g_{(z_n, \alpha_n)}, \quad \g_n^+ := \g_{(z_n,\beta_n)}.
\end{align*}
We also denote by $g$ the gedesic from $-1$ to $-i$, i.e.,  $g:=\g_{(-1,-i)}$.

Below, as in Section \ref{section:gen-teich-rays-domains}, we denote by $\tau_{\eps,W}$ the point in $T(\mathbb{D})$ corresponding to the vertical compression of $W$ by $\eps>0$, i.e.,
$$\tau_{\eps,W}=\left[ \left(\frac{1-\eps}{1+\eps}\right) \frac{|\varphi_{W}|}{\varphi_{W}} \right].$$

The following result states that the generalized Teichm\"uller geodesic ray $\eps\mapsto \tau_{\eps,W},$  does not converge to any point in Thurston boundary, and, moreover, it describes the limit set of the ray in Thurston boundary $\partial_{\infty} T(\mathbb{D})$, i.e.,  all the possible points in  $PLM_{bdd}(\mathbb{D})$ which occur as weak star limits of the geodesic ray $\eps\mapsto \tau_{\eps,W}.$

We recall that given a hyperbolic geodesic $\g=\g_{(x,y)}$ in $\mathbb{D}$ we will denote by $\delta_{\g}$ the Dirac mass on $\g$. In particular,  if $B=(a,b)\times (c,d)\subset G(\mathbb{D})$ is a box of geodesics, we will have that
 \begin{align*}
 \delta_{\g}(B) =
 \begin{cases}
 0,  \mbox{ if }  \g\in B,\\ 
 1,  \mbox{ otherwise.}
 \end{cases}
 \end{align*}
 
 Recall, that $g=\g_{(-1,-i)}$ and $\d_g$ is the Dirac mass on the geodesic connecting $-1$ and $-i$ in $\mathbb{D}$.  We will also denote  $ \lambda_W = \sum_{n=0}^{\infty} (\delta_{\g_n^+} +\delta_{\g_n^{-}})$. Thus, $\lambda_W$ is a geodesic measured lamination supported on the geodesics connecting $z_n$ to $\alpha_n$ and $\beta_n$ for $n\in\mathbb{N}$.

In what follows we use the notation 
\begin{align*}
A_n &= \prod_{k=0}^n a_k,\quad \mbox{ and } \quad 
B_n=\prod_{k=0}^n b_k=\prod_{k=1}^n b_k.
\end{align*}

\begin{theorem}\label{thm:main}
Let $W$ be a domain defined as in (\ref{domain}) so that there is constant $0<c<1$ such that for all $n\in\mathbb{N}$ we have 
\begin{align}\label{limit:ratio}
\max\left({b_{n+1}}/{a_n},{a_n}/{b_n}\right)<c<1
\end{align}
and 
\begin{align}\label{limit:roots}
\max \left\{ \sqrt[n]{a_n},\sqrt[n]{b_n}\right\}\underset{n\to\infty}\to0.
\end{align}
Then the limit set $\Lambda$ of the generalized Teichm\"uller ray $\eps\mapsto \tau_{\eps,W}$ in $PML_{bdd}(\mathbb{D})$ can be described as follows:
\begin{align}\label{limitset}
\Lambda = \{\llbracket s \delta_{g} + (2/3)\lambda_W \rrbracket : s\in[m,M]\},
\end{align}
where
\begin{align}\label{m&M:1}
\begin{split}
m&=1+\liminf_{n\to\infty}\frac{\log (A_{n-1}/B_n)}{\log(1/a_n)}, \quad  M=1+\limsup_{n\to\infty}\frac{\log (A_{n-1}/B_n)}{\log(1/b_n)}.
\end{split}
\end{align}
\end{theorem}
Note that all the geodesic measured laminations $s \delta_{g} + \lambda_W$ have the same support and they only differ by the weight on the geodesic $g$.  
\begin{proof}[Proof of Theorem \ref{thm:main-intro}]
Letting  $x_{2n}=b_n$ and $x_{2n+1}=a_n$ for $n\geq 0$, with $x_0=1$  we see that if $x_{n+1}/x_n\to0$  then  (\ref{limit:ratio}) holds for every $c>0$.  Therefore, for every $\eps\in (0,1)$ there is an $N\in\mathbb{N}$ s.t.  for $n\geq N$ we have $x_{n+2}/x_n<\eps$ and hence $x_{N+2m}\leq \eps^m x_N$.  Thus,  if $m$ is so large that $n:=N+2m<3m$ and $x_N^{1/n}<2$  then
$x_n^{{1}/{n}} \leq \eps^{\frac{m}{N+2m}} x_N^{1/n} \leq 2\eps^{1/3}$. Hence $\sqrt[n]{x_n}\to0$ as $n\to\infty$, which implies (\ref{limit:roots}).  By the definition of $x_n$ above we have that (\ref{m&M:1-intro}) s equivalent to (\ref{m&M:1}), which completes the proof.
\end{proof}

\begin{remark}
\rm{Since $\log\frac{A_{n-1}}{B_n} = \log \frac{A_n}{B_n} +\log \frac{1}{a_n}=\log \frac{A_{n-1}}{B_{n-1}} +\log \frac{1}{b_n}$ we have that
\begin{align}\label{m&M:2}
m=2-\limsup_{n\to\infty}\frac{\log (B_n/A_n)}{\log(1/a_n)}, \quad 
M=2-\liminf_{n\to\infty}\frac{\log(B_n/A_n)}{\log (1/b_{n+1})}.
\end{align}
Therefore, since all the limit terms in (\ref{m&M:1}) and (\ref{m&M:2}) are non-negative, we have that {$1\leq m<M\leq 2$.}  In particular, every limit set $\Lambda$  as in (\ref{limitset}) corresponds to some  interval $[m,M]\subset[1,2]$. Below we will provide specific examples of sequences $\{a_n\}$ and $\{b_n\}$ showing that every interval $[m,M]\subset(1,2)$  occurs as a limit set.  It is also possible to have $m=1$ and $M=2$, but we  leave such constructions to the reader.
}
\end{remark}

Next we provide some examples of geodesics rays which are either convergent or divergent at $T_{\infty}(\mathbb{D})$.
 
\begin{example}\label{example:factorial}
\rm{
Suppose $b_0=1$,  $a_0=1/2$, $b_1=a_0/3$, $a_1=b_1/4$, etc. Thus
\begin{align*}
a_n=\frac{1}{(2n+2)!},  \quad b_n =\frac{1}{(2n+1)!}.
\end{align*}
Then $b_n/a_n=2(n+1)$,    $	{B_n}/{A_n} =  2^{n+1}(n+1)! = (2n+2)!!$, and therefore
\begin{align*}
\frac{\log\frac{B_n}{A_n}}{{\log\frac{1}{a_n}}} = \frac{\log (2n+2)!!}{\log (2n+2)!} = \frac{\log (2n+2)!!}{\log (2n+2)!! +\log(2n+1)!!} \underset{n\to\infty}{\longrightarrow} \frac{1}{2}.
\end{align*}
Since $\log{a_n}/\log{b_n}\to1$, we have that $\log(B_n/A_n)/\log(1/b_{n+1}) \to 1/2$ as well.  Therefore $m=M=3/2$ and the generalized Teichm\"uller geodesic  converges to 
$\left\llbracket \frac{3}{2} \delta_{g} + \frac{2}{3}\lambda_W \right\rrbracket$.
}
\end{example}

Generalizing this example slightly,  we obtain the following. 


\begin{example}\label{example:risingfactorial}
	For every $s \in(1, 2)$ there is a domain $W$ as above so that
	\begin{align*}
		\tau_{\eps,W} \underset{\eps\to0}{\longrightarrow} \llbracket s \delta_{g} + (2/3)\lambda_W \rrbracket.
	\end{align*}
\end{example}

\begin{proof}
For all nonnegative integers $x$, we define the \emph{rising factorial} $(x)^{\bar n}$ by $$(x)^{\bar n} = x \cdot (x+1) \cdot ... \cdot(x+n-1).$$
Fix integers $p, q$ and $r$ with $1 \leq p, q < r$.  Suppose $b_0=1$. Let
\begin{align*}
a_n&=\frac{(nr+1)^{\bar p}}{(nr+1)^{\bar r}} \ b_n, \text{ for } n \geq 0, \\
b_n&=\frac{(nr+1)^{\bar q}}{(nr+1)^{\bar r}} \ a_{n-1}, \text{ for } n \geq 1.
\end{align*}
Thus $\displaystyle a_n = \prod_{k=1}^{n} \frac{(kr+1)^{\bar p} (kr+1)^{\bar q}}{(kr+1)^{\bar r}(kr+1)^{\bar r}}$ and
		$\displaystyle b_n = \prod_{k=1}^{n} \frac{((k-1)r+1)^{\bar p} (kr+1)^{\bar q}}{((k-1)r+1)^{\bar r}(kr+1)^{\bar r}}$.
Then $\displaystyle \frac{b_n}{a_n}=\frac{(nr+1)^{\bar r}}{(nr+1)^{\bar p}}$, $\displaystyle \frac{B_n}{A_n}=\prod_{k=1}^{n}\frac{(kr+1)^{\bar r}}{(kr+1)^{\bar p}}$, and therefore
		\begin{align*}
			\frac{\log\frac{B_n}{A_n}}{{\log\frac{1}{a_n}}} 
			&= \frac{\log\displaystyle\prod_{k=1}^{n}\frac{(kr+1)^{\bar r}}{(kr+1)^{\bar p}}}{\log\displaystyle\prod_{k=1}^{n} \frac{(kr+1)^{\bar r} (kr+1)^{\bar r}}{(kr+1)^{\bar p}(kr+1)^{\bar q}}} = \frac{1}{1+\frac{\log\displaystyle\prod_{k=1}^{n}\frac{(kr+1)^{\bar r}}{(kr+1)^{\bar q}}}{\log\displaystyle\prod_{k=1}^{n}\frac{(kr+1)^{\bar r}}{(kr+1)^{\bar p}}}} \\
			& \underset{n\to\infty}{\longrightarrow} \frac{1}{1+\frac{r-q}{r-p}}=\frac{r-p}{2r-p-q} \in (0, 1).
		\end{align*}
		Since $\displaystyle\frac{\log{a_n}}{\log{b_{n+1}}} \underset{n\to\infty}{\longrightarrow} 1$, we have that $\displaystyle\frac{\log \frac{B_n}{A_n}}{\log \frac{1}{b_{n+1}}}  \underset{n\to\infty}{\longrightarrow} \frac{r-p}{2r-p-q}$ as well.  Therefore $m=M=\displaystyle 2-\frac{r-p}{2r-p-q} \in (1, 2)$. By choosing appropriate $p, q$ and $r$, we obtain all rational $s \in (1, 2)$. Hence, the generalized Teichm\"uller geodesic converges to 
		$\left\llbracket s \delta_{g} + \frac{2}{3}\lambda_W \right\rrbracket$.
\end{proof}


Next we construct the first example of a divergent geodesic.
\begin{example}\label{example:p}
\rm{Suppose $p>1$, $b_0=1$, $a_0=a\in(0,1)$ and for $n\geq 1$ we have $a_n=b_n^p$ and $b_{n+1}=a_n^p$. Then $a_n=a^{p^{2n}}$ and $b_{n+1}=a^{p^{2n+1}}$. Hence,  $$A_n=B_n^{p} = a^{\frac{p^{2n+2}-1}{p^2-1}}=Ca_n ^{\frac{p^2}{p^2-1}},$$  
where $C=a^{-\frac{1}{p^2-1}}$.  Therefore,
\begin{align}\label{example:p-a_n}
\begin{split}
m&=2-\lim_{n\to\infty}\frac{\log \frac{B_n}{A_n}}{\log\frac{1}{a_n}} =2- \lim_{n\to\infty} \frac{(\frac{1}{p}-1)\log A_n}{\log\frac{1}{a_n}}\\
& =2- \left(1-\frac{1}{p}\right)\frac{p^2}{p^2-1}  = 2-\frac{p}{p+1}=1+\frac{1}{p+1}.
\end{split}
\end{align}
Similarly
\begin{align}\label{example:p-b_n}
\begin{split}
M&=2-\lim_{n\to\infty}\frac{\log \frac{B_n}{A_n}}{\log\frac{1}{b_{n+1}}} =2- \lim_{n\to\infty} \frac{(\frac{1}{p}-1)\log A_n}{\log\frac{1}{a_n^p}}\\
& =2- \left(1-\frac{1}{p}\right)\frac{p}{p^2-1}  = 2-\frac{1}{p+1}.
\end{split}
\end{align}
Therefore, if $p>1$ then the limit set $\Lambda$ in this case is given by the formula (\ref{limitset}), where $[m,M]= \left[1+\frac{1}{p+1},2-\frac{1}{p+1}\right]\subset(1,2).$
In particular, $\Lambda$ is not a point, and geodesic ray $\eps\mapsto\tau_{\eps,W}$ diverges at infinity.
}
\end{example}

More generally, we have the following.

\begin{example}
For every nontrivial interval $I=[1+\alpha,1+\beta]\subset (1,2)$ there is a domain $W$ as above so that the corresponding generalized Teichm\"uller geodesic $\eps\mapsto \tau_{\eps,W}$  diverges and its limit set $\Lambda$ in $PML_{bdd}(\mathbb{D})$ is given by
\begin{align*}
\Lambda=\{ \llbracket s \delta_{g} + (2/3)\lambda_W \rrbracket: s\in I\}.
\end{align*}
\end{example}
\begin{proof}
\rm{
Suppose $p,q>1$, $b_0=1$, $a_0=a\in(0,1)$ and for $n\geq 1$ we have $a_n=b_n^p$ and $b_{n+1}=a_n^q$.  Then 
$a_n=a^{p^nq^n}$, $b_n=a^{p^nq^{n-1}}$,  and therefore
\begin{align*}
A_n&=a^{\frac{(pq)^{n+1}-1}{pq-1}}\asymp a_n^{\frac{pq}{pq-1}},\\
B_n&=a^{\frac{p((pq)^{n}-1)}{pq-1}}\asymp a_n^{\frac{p}{pq-1}}.
\end{align*}
Therefore,  since $A_n=B_n^p$ we have 
\begin{align*}
m&=2-\lim_{n\to\infty}\frac{\log \frac{B_n}{A_n}}{\log\frac{1}{a_n}} =2- \lim_{n\to\infty} \frac{(\frac{1}{p}-1)\log A_n}{\log\frac{1}{a_n}}\\
& =2- \left(1-\frac{1}{p}\right)\frac{pq}{pq-1}  = 1+\frac{q-1}{pq-1}.
\end{align*}

Similarly
\begin{align*}
M&=2-\lim_{n\to\infty}\frac{\log \frac{B_n}{A_n}}{\log\frac{1}{b_{n+1}}} =2- \lim_{n\to\infty} \frac{(\frac{1}{p}-1)\log A_n}{\log\frac{1}{a_n^q}}\\
& =2- \left(1-\frac{1}{p}\right)\frac{p}{pq-1}  = 2-\frac{p-1}{pq-1}=1+ \frac{p(q-1)}{pq-1}.
\end{align*}
Therefore, 
\begin{align*}
[m,M]=\left[1+\frac{q-1}{pq-1},1+\frac{p(q-1)}{pq-1}\right]\subset (1,2).
\end{align*}
Observe, that if $0<\alpha<\beta<1$ then letting $p=\frac{\beta}{\alpha}$ and $q=\frac{1-\alpha}{1-\beta}$ we obtain that $m=1+\alpha$ and $M=1+\beta$.  Therefore a limit set of the generalized Teichm\"uller geodesic can be given by (\ref{limitset}) where $[m,M]$ is any interval in $(1,2)$.
}
\end{proof}

\section{Proof of Theorem \ref{thm:main}}\label{section:independence}
To prove that the right hand side is a subset of the left hand side in (\ref{limitset}) of  Theorem \ref{thm:main} it is enough to show that for every $s\in[m,M]$ there is a sequence $n_k=n_k(s)$ such that 
\begin{align}\label{thm:limits}
\tau_{\eps_{n_k},W} \underset{k\to\infty}\longrightarrow\llbracket s \delta_{g} + (2/3)\lambda_W \rrbracket\in\Lambda.
\end{align}
For the opposite inclusion we need to show that every limit point of $\tau_{\eps,W}$ as $\eps\to0$ is of the form $\llbracket s \delta_{g} + (2/3)\lambda_W \rrbracket$ for some $s\in[m,M]$.


By Lemma \ref{lemma:discrete-laminations},  (\ref{thm:limits}) would follow if we showed that for every $\g\in G(\mathbb{D})$ there is a constant $\eta_{\g}>0$ and a sequence $\{n_k\}$ so that for $\eta<\eta_{\gamma}$ we have the following:
\begin{align}\label{modlimit-1}
\lim_{k\to\infty}\frac{\m V_{\eps_{n_k}}(\phi(\G_{B_{\eta}(\g)}))}{\frac{1}{\pi}\log\frac{1}{\eps_{n_k}}}=
\begin{cases}
s, &\mbox{ if } \g=g,\\
\frac{2}{3}, &\mbox{ if } \g\in\bigcup_{n=0}^{\infty}\{\g_n^+,\g_n^-\} ,\\
0, &\mbox{otherwise},
\end{cases}
\end{align} 
where $s\in[m,M]$.  For $\g\neq g$  equality (\ref{modlimit-1}) follows from the proof of  \cite[Theorem 5.1]{HakSar-visual}.  Thus, to prove (\ref{thm:limits}) it is enough to assume $\g=g=\g_{(-1,-i)}$.  To simplify the notation for every curve family $\G$ in $\mathbb{D}$ we denote 
$\G^{\eps}=V_{\eps}(\phi(\G)).$ In particular,  if $I,J$  are disjoint arcs on the unit circle $\partial\mathbb{D}$ and  $\G_{I,J} = \G(I,J;\mathbb{D})$ is the curve family connecting $I$ to $J$ in $\mathbb{D}$, we denote
\begin{align*}
\G^{\eps}_{I,J} =\G^{\eps}_{I\times J} = V_{\eps}(\phi(\G_{I,J})).
\end{align*}
The following lemma  implies (\ref{modlimit-1}) in the case of $\g=g$ and therefore also proves Theorem \ref{thm:main}.

For every pair of disjoint arcs $I$ and $J$ on $\mathbb{S}^1$ we define
\begin{align}
\mathbb{M}_{I,J}(\eps):=\frac{\m \G^{\eps}_{I,J}}{\frac{1}{\pi} \log \frac{1}{\eps}}.
\end{align}

\begin{lemma}\label{lemma:limits}
For every  $s\in[m,M]$ there is a sequence $\eps_n\in[a_{n+1},a_n]$ for which  
\begin{align}\label{eqn:modlimits}
\lim_{n\to\infty}\mathbb{M}_{I,J}(\eps_n) = s,
\end{align} 
whenever $I,J$  are disjoint arcs on $\mathbb{S}^1$ s. t.   $g\in I\times J$ and  $1\notin I\cup J$.
\end{lemma} 

Note that the condition on the arcs $I$ and $J$ can be formulated more concretely as follow.  Assuming that $-1\in I$ and $-i\in J$ we have $J\subset \{\Im (z) <0\}$,  and $1\notin I$.

Before proceeding, we explain how Lemma \ref{lemma:limits}  implies equality (\ref{modlimit-1}) for $\g=g$.  Note that for every  $\eta>0$ small enough we can find boxes of geodesics $B'=I'\times J'$ and $B''=I''\times J''$ satisfying the conditions of Lemma  \ref{lemma:limits}, such that 
$B'\subset B_{\eta}(g) \subset B''$. Therefore, by monotonicity of modulus we have
$\m \G^{\eps}_{B'}  \leq \m \G^{\eps}_{B_{\eta}(g)} \leq \m \G^{\eps}_{B'} $ and since the limit  (\ref{eqn:modlimits}) is the same for  the boxes $B'$ and $B''$ it would have to be the same for $B_{\eta}(g)$.

\begin{figure}\label{figure:intervals}
	\centerline{
		\includegraphics[width=0.6\textwidth]{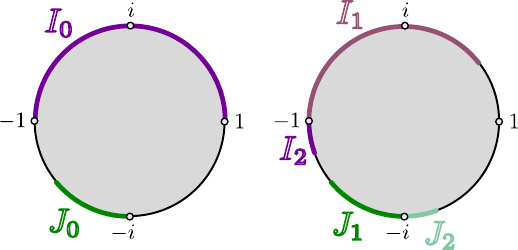}}
	\caption{\small{Estimating $\mathrm{mod} \Gamma_{k, 2}^{\eps}$ from above}.}
\end{figure}

\begin{proof}[Proof of Lemma \ref{lemma:limits}]
We start by showing  that the limit (\ref{eqn:modlimits}) is independent of the arcs $I$ and $J$. Let
\begin{align}\label{notation:intervals}
 I_0&=\{|z|=1\} \cap \{\Im(z)>0\}, \quad J_0=\phi^{-1}(\{0\} \times (-\infty,-1)).
\end{align}

\begin{proposition}\label{lemma:independence}
If $I$ and $ J$ are disjoint arcs of $\mathbb{S}^1$ s.t.  $g\in I\times J$ and  $1\notin I\cup J$ then for a sequence $\eps_k\to0$ we have
\begin{align*}
\lim_{k\to \infty}\mathbb{M}_{I,J}(\eps_k) =\lim_{k\to \infty}\mathbb{M}_{I_0,J_0}(\eps_k) 
\end{align*}
whenever the limits above exist.
\end{proposition}

\begin{proof}[Proof of Proposition \ref{lemma:independence}]
It is enough to show that there is a constant $C\geq 1$ such that 
\begin{align}\label{mod:comparable}
|\m \G^{\eps}(I,J) - \m \G^{\eps}(I_0,J_0)|\leq C
\end{align}
for $\eps\to0$. 
Let $Q_1$ be the first quadrant in the plane, i.e., $Q_1 = \{z\in \mathbb{C} : \Re z>0, \Im z>0\}$,  and let ${Q_2,Q_3, Q_4}$ denote the other quadrants ordered counterclockwise.  Define
\begin{align*}
I_1= I\cap (Q_1\cup Q_2), \quad I_2= I\cap Q_3, \quad
J_1=J\cap Q_3,  \quad J_2= J\cap Q_4.
\end{align*}
Since $\{1,-i\} \cap I =\emptyset$ and $\{-1,1\}\cap J =\emptyset$ we have that $I=I_1\cup I_2$ and $J=J_1\cup J_2$. 
Denoting as before $\G^{\eps}(E,F) = V_{\eps}(\phi(E), \phi(F); W)$, whenever $E,F\subset \mathbb{S}^1$ we have
\begin{align}\label{mod:bound1}
\G^{\eps}(I,J) \subset \G^{\eps}(I_1,J_1) \cup \G^{\eps}(I_1,J_2) \cup \G^{\eps}(I_2,J_1) \cup \G^{\eps}(I_2,J_2).
\end{align}
To simplify the notation, given an arc $I\subset \mathbb{S}$ we let $\tilde{I}$ be the subset of $\partial{W}$ so that $\phi^{-1}(\tilde{I})=I$, and $\tilde{I}^{\eps} = V_{\eps}(\tilde{I})$, where abusing the notation slightly by $\phi^{-1}$ we denote the extension of $\phi^{-1}$ to $\overline{W}$.  Note that since $\partial{W}$ is not locally connected at any point $ia$, with $a>0$,  we have that  if $0\in I$ we cannot write $\tilde{I}=\phi(I)$ since $\phi$ does not extend continuously to $0\in\partial{\mathbb{D}}$.  However if $0\notin I$ then $\tilde{I}=\phi(I)$.

Note that $\G^{\eps}(I_1,J_2)$ overflows $\G([0,1],\tilde{J}_2^{\eps};W)$.  Since $\tilde{J}_2^{\eps}=\phi(J_2)$ we have that the relative distance between  $[0,1]$ and $\tilde{J}_2^{\eps}$ is bounded below and therefore by Lemma   \ref{lemma:mod-reldist} there is a $c_1>0$ such that 
$\m\G^{\eps}(I_1,J_2) \leq \G([0,1],\tilde{J}_2^{\eps};W) \leq c_1.$
 
 Since $\phi(I_2)$ and $\phi(J_1)$ both belong to the imaginary axis,  relative distance between the $\tilde{I}_2^{\eps}$ and $\tilde{J}_1^{\eps}$ does not change with $\epsilon$ and therefore is bounded below. Thus $\m \G^{\eps}(I_2,J_1) \leq c_2 $ for some $c_2>0$. Finally, $\Delta(I_2^{\eps},J_2^{\epsilon})\to\infty$ since $\diam (I_2^{\eps})\to0$ as $\eps\to0$.

Thus,  the moduli of all the families on the right hand side of (\ref{mod:bound1}) are bounded above as $\eps\to0$, except for $\m \G^{\eps}(I_1,J_1)$,  and  as $\eps\to0$ we have
\begin{align*}
\m  \G^{\eps}(I_1,J_1) \leq \m \G^{\eps}(I,J) \leq \m \G^{\eps}(I_1,J_1) + C_1
\end{align*}

Therefore, to prove (\ref{mod:comparable}) it is enough to show that there is a constant $C\geq 1$ such that 
\begin{align}\label{mod:comparable2}
|\m \G^{\eps}(I_1,J_1) - \m \G^{\eps}(I_0,J_0)|\leq C.
\end{align}

Since $I_1\subset I_0$ we have 
\begin{align*}
\G^{\eps}(I_1,J_1) 
&\subset \G^{\eps}(I_0,J_1\cap J_0) \cup \G^{\eps}(I_0,J_1\setminus J_0)\subset \G^{\eps}(I_0, J_0) \cup \G^{\eps}(I_0,J_1\setminus J_0)
\end{align*}
Observe that since $\tilde{J}_1\setminus \tilde{J}_0$ is either empty or an interval compactly contained in $\{0\}\times (-\infty,0)$,   we have that 
$\Delta(I_0,\tilde{J}_1^{\eps}\setminus \tilde{J}_0^{\eps})$ is independent of $\eps$ and thus is bounded below away from zero.  Hence, by Lemma \ref{lemma:mod-reldist} we have 
\begin{align}
\m\G^{\eps}(I_1,J_1) \leq \m\G^{\eps}(I_0,J_0) +C.
\end{align}
On the other hand,  we also have the inclusions 
\begin{align*}
\G^{\eps}(I_0,J_0)
&\subset \G^{\eps}(I_0\cap I_1, J_0) \cup  \G^{\eps}(I_0\setminus I_1, J_0) \\
&\subset \G^{\eps}( I_1, J_0\cap J_1)\cup \G^{\eps}( I_1, J_0\setminus J_1) \cup  \G^{\eps}(I_0\setminus I_1, J_0)\\
&\subset \G^{\eps}( I_1, J_1)\cup \G^{\eps}( I_1, J_0\setminus J_1) \cup  \G^{\eps}(I_0\setminus I_1, J_0).
\end{align*}
Just like above, since $\tilde{J}_0\setminus \tilde{J}_1$ is compactly contained in $\{0\}\times (-\infty,0)$,  we have that $\Delta(I_1,\tilde{J}_0^{\eps}\setminus \tilde{J}_1^{\eps})$ is bounded away from $0$ and hence the modulus of $\G^{\eps}(I_1, J_0\setminus J_1)$ is bounded above by a constant for all $\eps>0$.

Finally,  let $N$ be the smallest natural number $n$ so that $z_n\in I_1$.  From the construction it follows that  $\G^{\eps}(I_0\setminus I_1,J_0)$ overflows $\G([a_N,1], \tilde{J}_0^{\eps})$. Since $\Delta([a_N,1],\tilde{J}_0^{\eps}) \geq \frac{a_N}{1-a_N}>0$ for all $\eps>0$, we again have that $\m \G^{\eps}(I_0\setminus I_1,J_0)$ is bounded above independently of $\eps$.  Therefore, using the inclusions above we obtain
\begin{align}
\m\G^{\eps}(I_0,J_0) \leq \m\G^{\eps}(I_1,J_1) +C,
\end{align}
for some constant $C$ and all $\eps>0$. This proves (\ref{mod:comparable2}), which in turn implies (\ref{mod:comparable}) and Proposition \ref{lemma:independence}
\end{proof}

Continuing with the proof of Lemma \ref{lemma:limits} we observe that $\mathbb{M}_{I_0,J_0}(\eps)$ is a continuous function of $\eps$ in $(0,1)$.  This follows from Caratheodory's theorem on continuous extension of Riemann mapping at the points where $\partial W$ is locally connected. However, for readers convenience we give a direct proof.  From the definition of $\mathbb{M}_{I_0,J_0}(\eps)$ it is enough to show that $\m \G^{\eps}_{I_0,J_0}$ is continuous.  Note that 
$ \m \G^{\eps}_{I_0,J_0} = \m \G (V_{\eps}(\phi(I_0)), V_{\eps}(\phi(J_0));W) = \m\G(\phi(I_0),\{0\}\times (-\infty,-\eps);W).$
Therefore if $0<\zeta<\eps$ then $\G^{\eps}_{I_0,J_0}\subset \G^{\zeta}_{I_0,J_0}$ then
$$\m \G^{\eps}_{I_0,J_0} \leq \m \G^{\zeta}_{I_0,J_0}\leq \m \G^{\eps}_{I_0,J_0} + \m (\G^{\zeta}_{I_0,J_0}\setminus \G^{\eps}_{I_0,J_0}). $$  
Let $\tau=(\eps-\zeta)/2.$ Then $\G^{\zeta}_{I_0,J_0}\setminus \G^{\eps}_{I_0,J_0}$ overflows the family of the
curves connecting the boundary components of the annulus $A(-i\frac{\zeta+\eps}{2}, \tau,\eps-\tau)$ in $W$. Therefore, as $\zeta \to \eps$ or, equivalently, if $\tau\to0$, we have 
\begin{align*}
|\m\G^{\zeta}_{I_0,J_0}-\m \G^{\eps}_{I_0,J_0}| \leq   {\pi}\left({\log \frac{\eps-\tau}{\tau}}\right)^{-1} \underset{\tau\to0}{\longrightarrow}0.
\end{align*}
Hence ${\m \G^{\eps}_{I_0,J_0}}/({\frac{1}{\pi} \log \frac{1}{\eps}})$ is continuous in $\eps$. 

Since $\mathbb{M}_{I_0,J_0}(\eps)$ is continuous,  every $s$ which is bounded between the inferior and superior limits of $\mathbb{M}_{I_0,J_0}(\eps)$ as $\eps\to0$
is a subsequential limit of $\mathbb{M}_{I_0,J_0}(\eps)$.  By Proposition \ref{lemma:independence} the same would also hold for $\mathbb{M}_{I,J}(\eps)$, whenever $I$ and $J$ satisfy the conditions of Lemma \ref{lemma:limits}.
Therefore,   Lemma \ref{lemma:limits} follows from Lemma \ref{lemma:limsup&liminf} which states that the  inferior and superior limits of $\mathbb{M}_{I_0,J_0}(\eps)$ as $\eps$ approaches $0$ are equal to $m$ and $M$,  respectively.
\end{proof}

\section{Modulus bounds}\label{section:mod-bounds}
By Proposition \ref{lemma:independence} a key step in proving Lemma \ref{lemma:limits} and Theorem \ref{thm:main} are bounds on $\mathbb{M}_{I_0,J_0}(\eps)$, which are established in this section.  To simplify the notation we let  
\begin{align}
\G=\phi(\G(I_0,J_0;\mathbb{D}))=\G(\phi(I_0),\phi(J_0);W),
\end{align} and for every $\eps>0$ we let $\G^{\eps}=V_{\eps}(\G)$. Since $V_{\eps}(\phi(I_0))=\phi(I_0)$ we have 
\begin{align}
\G^{\eps}
&=\G(\phi(I_0),\{0\}\times(-\infty,-\eps);W).
\end{align}
Furthermore,  given $z\in\mathbb{C}$,  and $0<r<R<\infty$ we will denote
$$A(z,r,R) = {B(z,R)}\setminus \overline{B(z,r)}.$$ 

To simplify the exposition we denote
\begin{align}\label{def:M_n&m_n}
\begin{split}
m_n&=1+\frac{\log ({A_{n-1}}/{B_{n}})}{\log({1}/{a_n})}=2-\frac{\log (B_n/A_n)}{\log(1/a_{n})}\\
M_n&=1+\frac{\log ({A_{n-1}}/{B_{n}})}{\log({1}/{b_n})}=2-\frac{\log (B_{n-1}/A_{n-1})}{\log(1/b_{n})}
\end{split}
\end{align}
Therefore, by Lemma \ref{lemma:limits} and Proposition \ref{lemma:independence}, Theorem \ref{thm:main} would follow from the following result.
\begin{lemma}\label{lemma:limsup&liminf}
With the notation above we have the following equalities
\begin{align}
 \liminf_{\eps\to0} \mathbb{M}_{I_0,J_0}(\eps)&=\liminf_{n\to\infty}m_n \label{eq:liminf}=m,\\
 \limsup_{\eps\to0} \mathbb{M}_{I_0,J_0}(\eps)&=\limsup_{n\to\infty}M_n\label{eq:limsup}=M.
\end{align}
\end{lemma}
The rest of this section is devoted to proving  Lemma \ref{lemma:limsup&liminf}.

\subsection{Lower bounds for $\mathbb{M}_{I_0,J_0}(\eps)$}

\begin{figure}\label{figure:lower}
	\centerline{
		\includegraphics[width=1\textwidth]{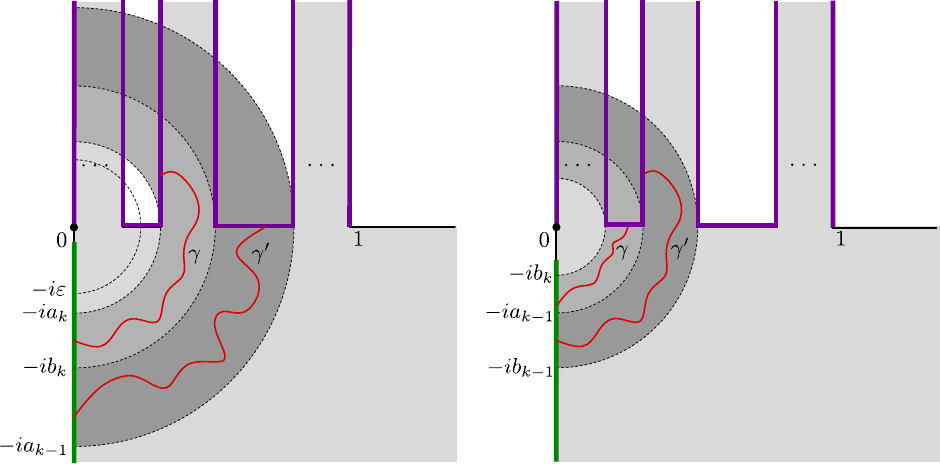}}
	\caption{\small{Estimating $\mathrm{mod} \Gamma^{\eps}$ from below}.}
\end{figure}

\begin{proposition}\label{prop:lower-bounds}
For $n\geq 1$ we have 
\begin{align}\label{lower-bound}
\begin{split}
\mathbb{M}_{I_0,J_0}(\eps)&\geq  2-\frac{\log{B_{n}}/A_n}{\log{1}/{\eps}},  \quad \mbox{for}\,\,  \eps\in [b_{n+1},a_n],\\
\mathbb{M}_{I_0,J_0}(\eps)&\geq  1+\frac{\log {A_{n-1}}/{B_{n}}}{\log{1}/{\eps}},  \quad \mbox{for} \,\, \eps\in [a_n,b_n].
\end{split}
\end{align}
In particular, for $\eps\in[b_{n+1},b_n]$ we have
\begin{align}\label{lower-bound:m_n}
\mathbb{M}_{I_0,J_0}(\eps)\geq m_n,
\end{align}
where $m_n$ is defined as in (\ref{def:M_n&m_n}).
\end{proposition}
\begin{proof}

For $k\geq 0$  we denote
\begin{align*}
G_{k,1}^{\eps} &=\{\g\in\G^{\eps} : \g\subset A(0, a_k,b_k)\},\\
G_{k,2}^{\eps} &=\{\g\in\G^{\eps} : \g\subset A(0, b_{k},a_{k-1})\}.
\end{align*}

	Suppose $\eps\in[b_{n+1},a_n]$.  Denote $G_{n}^{\eps} = \{\g\in \G^{\eps} | \g\subset A(0,\eps, a_n)\}$.  Using the motonicity and overflowing properties of modulus,  as well as the fact that the families $G_n^{\eps}$ and $G_{k,i}^{\eps}$ for $k\in\{0,\ldots,n\}$ and $i\in\{1,2\},$ are pairwise  disjoint, we have
\begin{align*}
\m\Gamma^{\eps} 
&\geq \m G_{n}^{\eps} + \sum_{k=0}^{n} (\m G_{k,1}^{\eps} + \m G_{k,2}^{\eps})\\
&\geq \frac{2}{\pi} \log \frac{a_n}{\eps}+\frac{1}{\pi} \sum_{k=0}^n \log \frac{b_k}{a_k}+ \frac{2}{\pi}\sum_{k=1}^n \log \frac{a_{k-1}}{b_k} \\
&= \frac{2}{\pi} \log \frac{a_n}{\eps}+\frac{1}{\pi}\log \frac{B_n }{A_n} + \frac{2}{\pi}\log \frac{A_{n-1}}{B_n}\\
&= \frac{2}{\pi}\log \frac{1}{\eps} -\frac{1}{\pi} \log \frac{B_n}{A_n}.
\end{align*}
Dividing both sides by $\frac{1}{\pi}\log\frac{1}{\eps}$ gives the first line in  (\ref{lower-bound}).  For $\eps\in[a_{n},b_{n}]$ we similarly estimate
\begin{align*}
\m\Gamma^{\eps} 
&\geq \frac{1}{\pi}\log\frac{b_{n}}{\eps}+\frac{2}{\pi}\log \frac{A_{n-1}}{B_{n}} + \frac{1}{\pi}\log \frac{B_{n-1}}{A_{n-1}}\\
&=\frac{1}{\pi}\log\frac{1}{\eps}+\frac{1}{\pi} \log \frac{A_{n-1}}{B_{n}}
\end{align*}
 Dividing by $\frac{1}{\pi}\log\frac{1}{\eps}$ gives the  second line in  (\ref{lower-bound}). 
 
Using (\ref{lower-bound}),  the fact that $1/(\log(1/\eps))$ is an increasing function, and the first line in (\ref{def:M_n&m_n}), we obtain (\ref{lower-bound:m_n}).
\end{proof}

\subsection{Upper bounds for $\mathbb{M}_{I_0,J_0}(\eps)$}
To obtain upper bounds for $\mathbb{M}_{I_0,J_0}(\eps)$ we need an auxilliary result.  As before we let
$\G^{\eps}=V_{\eps}(\G(\phi(I_0),\phi(J_0);W))=\G(\phi(I_0),\{0\}\times(-\infty,-\eps);W)$ for $\eps\in(0,1)$. Without loss of generality we may assume that $\g(0)$ belongs to the imaginary axis.  Thus $\Im(\g(0))<-\eps$.  We will denote by $\G_{k,1}^{\eps}$ and $\G_{k,2}^{\eps}$ the following subfamilies of $\G^{\eps}$:
\begin{align*}
\G_{k,1}^{\eps}&=\{\g\in\G^{\eps} : i\g(0) \in [a_k,b_k]\},  \quad k\geq 0,\\
\G_{k,2}^{\eps}&=\{\g\in\G^{\eps} :i \g(0) \in [b_k,a_{k-1}]\}, \quad k\geq 1,
\end{align*}
and let $\G_{0,2}^{\eps}=\{\g\in\G^{\eps}: i\g(0) \in  [1,\infty)\}$.  Note that if  $b_k<\eps$ then $\G^{\eps}_{k,1}$ is empty,  and if $a_{k-1}<\eps$ then $\G_{k,2}^{\eps}$ is empty.


\begin{lemma}\label{lemma:main}
There is a constant  $0<C<\infty$ depending on $c$ such that  
\begin{align}
 \m \G^{\eps}_{k,1} & \leq \frac{1}{\pi}\log\frac{b_k}{\max\{\eps,a_k\}} + C,  \,\, \mbox{\rm{for} } \eps < b_k,  \mbox{ \rm{and} } k\geq0, \label{modest:A_k} \\ 
\m \G^{\eps}_{k,2} & \leq \frac{2}{\pi}\log\frac{a_{k-1}}{\max\{\eps,b_{k}\}} + C,  \,\, \mbox{ \rm{for} } \eps <a_{k-1}, \mbox{ \rm{and} } k\geq1, \label{modest:B_k}\\
\m \G^{\eps}_{0,2} & \leq 2. \label{modest:3}
\end{align} 
\end{lemma}
\begin{proof}

\begin{figure}\label{figure:upper1}
\centerline{
\includegraphics[width=1\textwidth]{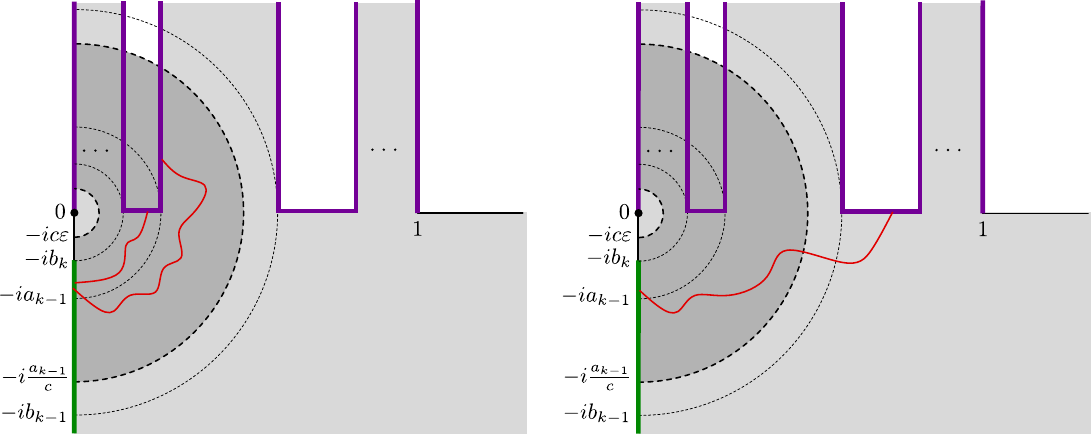}}
\caption{\small{Estimating $\mathrm{mod} \Gamma_{k, 2}^{\eps}$ from above}.}
\end{figure}

To obtain inequality (\ref{modest:3}) observe that if $\g\in \G_{0,2}$ then $\g$ contains a subcurve connecting  the interval $[0,1]$ and $\{z\in W : \mathrm{dist}(z,[0,1])=1\}\cap \{\Im (z)<0\}$.  Every such subcurve of $\g$ has length at least $1$ and is contained in the region $\{z \in W : \mathrm{dist}(z,[0,1])<1\} \cap \{\Im z<0\}$. The latter region has area $1+\pi/4<2$ and therefore by the overflowing property we obtain (\ref{modest:3}).

\textit{\underline{Proof of (\ref{modest:B_k}).}} Fix $\eps>0$ and  an integer $k\geq 1$ so that $\eps < a_{k-1}$.  If $\eps\geq b_k$ let 
$$\G_k' = \{\g\in \G_{k,2}^{\eps} : \g\subset A(0,c\eps,a_{k-1}/c)\}.$$
Note that every $\g\in \G_k'$ has a subcurve which separates the two boundary circles of the annulus $A(0,c\eps,a_{k-1}/c)$ and lies in the fourth quadrant $Q_4$.  Hence,
\begin{align*}
\m (\G_k') \leq \frac{2}{\pi}\log \frac{a_{k-1}/c}{c\eps} = \frac{2}{\pi}\left[\log \frac{a_{k-1}}{\eps} + 2\log \frac{1}{c}\right].
\end{align*}
Furthermore,  if $\g\in \G_{k,2}^{\eps} \setminus \G_k'$ then it  connects the two boundary components of either the annulus $A(0,a_{k-1},a_{k-1}/c)$ or the annulus $A(0,c\eps,\eps)$.  Therefore, since $\g\subset\{\Re z >0\}$, we have  $\m (\G_{k,2}^{\eps} \setminus \G_k') \leq  2 \cdot\frac{\pi}{\log (1/c)}.$
Since $\m \G_{k,2}^{\eps} \leq  \m \G_k' + \m(\G_{k,2}^{\eps} \setminus \G_k')$,  it follows that (\ref{modest:B_k}) holds with the constant  $C_1 = \frac{2}{\pi}\log \frac{1}{c^2}+ \frac{2\pi}{\log \frac{1}{c}}$ for $\eps\in[b_{k},a_{k-1}]$.

If $\eps<b_k$ then the same argument as above holds by replacing $\eps$ with $b_k$. Therefore we obtains
$\m \G^{\eps}_{k,2} \leq \frac{1}{\pi}\log\frac{a_{k-1}}{b_k} +C_1$ thus completing the proof of (\ref{modest:B_k}).



\textit{\underline{Proof of (\ref{modest:A_k}).}}  Fix $\eps>0$ and $k\geq0$ so that $\eps<b_k$.  If $\eps\geq a_k$ let 
$$\G_k'' = \{\g\in \G^{\eps}_{k,1}  :  \g \subset A(0, c\eps,b_{k}/c)\},$$ and $F_k =\G^{\eps}_{k,1} \setminus \G_k''.$
Just like above $\m F_k \leq 2 \cdot\frac{\pi}{\log (1/c)}.$

To bound from above $\m \G_k''$, we let $\zeta_k=(a_k,0)$ and consider the following subfamilies of $\G_k''$:
\begin{align*}
\mathscr{F}_k^1&=\{\g\in \G_k'' : \Im(\g(1)) \leq \eps\}, \\
\F_k^2&=\{\g\in \G_k'': \Im(\g(1)) \geq \eps,  \g \cap B(\zeta_k,(1-c)\eps)=\emptyset \}, \\
\F_k^3&= \G_k''\setminus (\mathscr{F}_k^1\cup\mathscr{F}_k^2).
	\end{align*}

\begin{figure}\label{figure:upper2}
\centerline{
\includegraphics[width=1\textwidth]{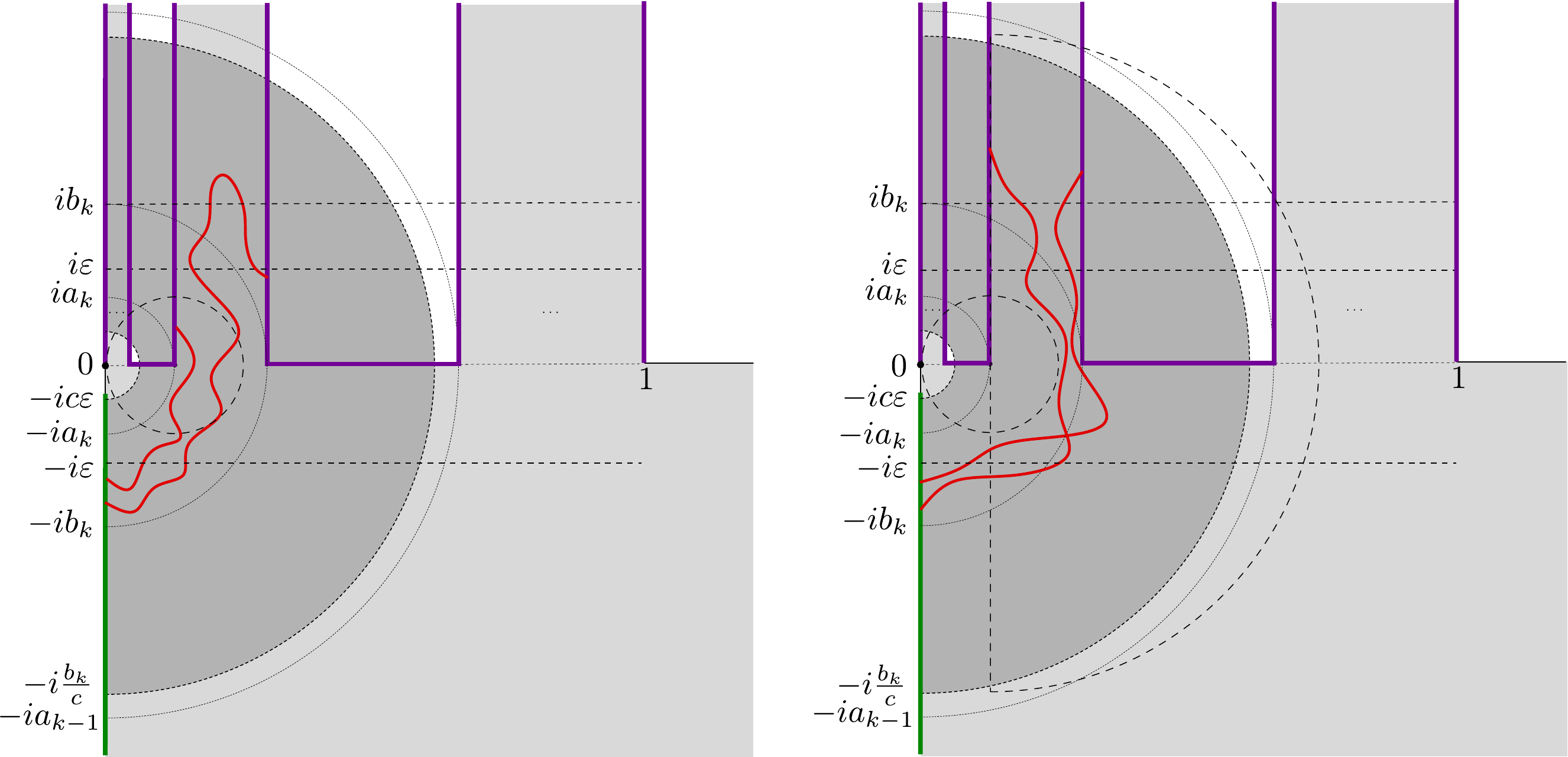}}
\caption{\small{Estimating $\mathrm{mod} \Gamma_{k, 1}^{\eps}$ from above}.}
\end{figure}

To estimate $\m\F_k^1$,  note that for every $\g\in\F_k^1$ we have either $\g(1)\in[ca_k,a_k]\cup(\{a_k\}\times[0,\eps])$ or not.  In the former case,  since $\Im(\g(0))\leq -\eps$ and $\Im(\g(1)) \geq 0$, we have that $\diam(\g)\geq \eps$.  Thus the modulus of the subfamily of curves $\g$ in $\F_k^1$ with $\g(1)\in[ca_k,a_k]\cup(\{a_k\}\times[0,\eps])$ can be estimated from above by 
\begin{align*}
\int_{N_{\eps}([ca_k,a_k]\cup(\{a_k\}\times[0,\eps])} \eps^{-2} dxdy <8\eps^2 \cdot \eps^{-2} = 8,
\end{align*} 
where $N_r(E)$ denotes the $r$-neghborhood of a set $E\subset\mathbb{C}$.  On the other hand, if  $\g(1)\notin[ca_k,a_k]\cup(\{a_k\}\times[0, \eps])$ then either $\g$ intersects the horizontal line $\{ \Im(z)=b_{k}\}$ or not.   If $\g\cap \{\Im(z)=b_k\} \neq \emptyset$ then $\g$ contains a subcurve which connects the horizontal sides of the square $[0,b_{k}]\times[0,b_{k}]$.  If $\g\cap\{\Im(z) =b_{k}\}=\emptyset$ then $\g$ contains a subcurve which  connects the vertical sides of the rectangle $ [0,b_{k}]\times[-c^{-1}b_{k},b_{k}]$.  Therefore,  using the overflowing property we obtain:
\begin{align*}
\m \F_k^1 \leq 8+1+ \frac{b_{k}-(-b_{k}/c)}{b_{k}}= 10+\frac{1}{c}.
\end{align*}


To estimate the modulus of $\F_k^2$,  we first  pick an arbitrary curve $\g\in\F_k^2$.  Note,  that since $b_{k+1}/a_k<c$, $b_k/a_{k-1}<c$, and $\g\subset A(0,c\eps, c^{-1}b_{k})$ we have that 
\begin{align*}
\g(1) \in (\{a_k\}\cup \{b_{k}\})\times (\eps,\infty).
\end{align*} 
In particular, $\Re(\g(1))\in \{a_k, b_k\}$.
If $\Re(\g(1))=b_{k}$ then like above we have two cases: either $\g\subset[0,b_{k}]\times[-b_{k}/c,b_{k}]$ or not.  As above in the former case $\g$ connects the vertical sides of $[0,b_{k}]\times[-b_{k}/c,b_{k}]$, and in the latter case it has a subcurve connecting the horizontal sides of the square $[0,b_{k}]\times[0,b_{k}]$.  Therefore the modulus of the family $\{\g\in\F_k^2 : \Re(\g(1))=b_{k}\}$ is bounded above by $\frac{b_k/c+b_k}{b_k}+1 = c^{-1}+2$. 

Next,  suppose $\g\in\F_k^2$ is such that $\Re(\g(1))=a_k$.   Recall, that  $\zeta_k=(a_k,0)\in\mathbb{R}$ and consider the annulus $$R_k=A(\zeta_k,(1-c)\eps,c^{-1}b_{k}).$$ 
Since $\g\cap B(\zeta_k, (1-c)\eps)=\emptyset$ it follows that  $\g$  intersects the vertical interval $\{a_k\} \times [-b_{k}/c,-(1-c)\eps]$ and $\g\cap\{\Re{z}>a_k\}$ is contained in $R_k$.   Therefore,  $\F_k^2$ overflows the family of curves in the semi-annulus $R_k\cap \{\Re z >a_k\}$ which separate the two boundary components of ${R}_k$.  The modulus of the latter family is $$\frac{1}{\pi}\log \frac{b_{k}/c}{(1-c)\eps} =\frac{1}{\pi}\log \frac{b_{k}}{\eps} +\frac{1}{\pi}\log \frac{1}{c(1-c)}.$$
Combining the two cases above we obtain
\begin{align*}
\m\F_k^2 \leq \frac{1}{\pi}\log\frac{b_{k}}{\eps}  +\frac{1}{\pi}\log \left(\frac{1}{c(1-c)}\right)+\frac{1}{c}+2.
\end{align*}

Finally,  $\F_k^3$ overflows the family of curves connecting the boundary components of the annulus $A(\zeta_k,(1-c)\eps,\eps)$,  which has modulus $\frac{2\pi}{\log \frac{1}{1-c}}$.

Since $\m \G_{k,1}^{\eps} \leq \m (\mathscr{F}_k^1 \cup \mathscr{F}_k^2\cup\mathscr{F}_k^3\cup F_k)$,  combining the estimates above we obtain (\ref{modest:A_k})
with $C_2=12+\frac{2}{c}+\frac{2\pi}{\log \frac{1}{c}} + \frac{2\pi}{\log \frac{1}{1-c}} +\frac{1}{\pi} \left[ \log\frac{1}{c} +\log\frac{1}{1-c}\right]$.  Therefore Lemma \ref{lemma:main} holds with $C=\max (C_1,C_2)$.
\end{proof}
We are now ready to estimate $\mathbb{M}_{I_0,J_0}(\eps)$ from above.

\begin{proposition}\label{prop:upper-bounds}
For $n\geq 1$ we have 
\begin{align}\label{upper-bound}
\begin{split}
\mathbb{M}_{I_0,J_0}(\eps)&\leq  2-\frac{\log {B_{n}}/{A_{n}}}{\log{1}/{\eps}}+o(1),  \quad \mbox{ for } \eps\in[b_{n+1},a_{n}]\\
\mathbb{M}_{I_0,J_0}(\eps)&\leq  1+\frac{\log {A_{n-1}}/{B_{n}}}{\log{1}/{\eps}}+o(1),  \quad \mbox{ for } \eps\in[a_n,b_n].
\end{split}
\end{align}
In particular, for $\eps\in[a_{n},a_{n-1}]$ we have
\begin{align}\label{ineq:limsup-upper-bound}
\mathbb{M}_{I_0,J_0}(\eps)\leq M_n+o(1),
\end{align}
where $M_n$ is defined in (\ref{def:M_n&m_n}).
 \end{proposition}

 \begin{proof} 
Suppose $\eps\in[b_{n+1},a_{n}]$,  for some $n\geq 1$.  Then $\G_{n
+1,2}^{\eps} = \{\g\in \G^{\eps} : i\g(0) \in [\eps,a_{n}]\}$. 
%
%
Therefore,  by Lemma \ref{lemma:main} there is a constant $C$ (not the same though as in estimates (\ref{modest:A_k}) and (\ref{modest:B_k})) so that for $n$ large enough we have
\begin{align*}
\m\G^{\eps}
& \leq \m \G^{\eps}_{n+1,2} + \sum_{k=0}^{n} \left( \m \G_{k,1}^{\eps} + \m \G_{k,2}^{\eps}\right) \\
& \leq \frac{2}{\pi}\log \frac{a_{n}}{\eps} + \frac{1}{\pi} \sum_{k=0}^{n} \log\frac{b_k}{a_k} +  \frac{2}{\pi} \sum_{k=1}^{n}\log\frac{a_{k-1}}{b_k} + C n\\
& \leq \frac{2}{\pi}\log \frac{1}{\eps} - \frac{1}{\pi}  \log\frac{B_{n}}{A_{n}} +Cn
\end{align*}
Since  $\eps\in[b_{n+1},a_{n}]$ we obtain
\begin{align*}
\mathbb{M}_{I_0,J_0}(\eps)
& \leq 2-\frac{\log(B_{n}/A_{n})}{\log(1/\eps)}+\frac{Cn}{\log(1/a_{n})}\leq 2-\frac{\log(B_{n}/A_{n})}{\log(1/\eps)}-\frac{C}{\log(\sqrt[n]{a_{n}})}
\end{align*}
Since $ \sqrt[n]{a_{n}}\to0$ as $n\to0$, we obtain (\ref{upper-bound}) for $\eps\in[b_{n+1},a_{n}]$.

Suppose  $\eps\in[a_n,b_n]$.  Then $\G_{n,1}^{\eps} = \{\g\in \G^{\eps} : i \g(0) \in [\eps,b_n]\}$.
Just like above,  by estimate (\ref{modest:B_k}) of Lemma \ref{lemma:main} there is a (possibly new)  constant $C$ such that 
 \begin{align*}
 \m \G^{\eps} 
 &\leq  \m \G_{n,1}^{\eps}   +\m \G_{n, 2}^{\eps}+\sum_{k=0}^{n-1} \left( \m \G_{k,1}^{\eps} + \m \G_{k,2}^{\eps}\right)\\
 &\leq \frac{1}{\pi} \log \frac{b_n}{\eps} +  \frac{1}{\pi} \sum_{k=0}^{n-1} \log\frac{b_k}{a_k} + \frac{2}{\pi}\sum_{k=0}^{n}\log\frac{a_{k-1}}{b_{k}}+Cn\\
 &=\frac{1}{\pi} \log \frac{1}{\eps} + \frac{1}{\pi} \log \frac{B_n}{A_{n-1}} + \frac{2}{\pi} \log\frac{A_{n-1}}{B_n} +Cn\\
 &=\frac{1}{\pi} \log \frac{1}{\eps} +  \frac{1}{\pi} \log\frac{A_{n-1}}{B_n} +Cn.
 \end{align*}
 Dividing both sides by $\frac{1}{\pi}\log(1/\eps)$ and using the fact that $\eps\leq b_n$ we obtain  
 \begin{align*}
 \mathbb{M}_{I_0,J_0}(\eps)
 &\leq 1+\frac{\log (A_{n-1}/B_n)}{ \log (1/\eps)} + \frac{Cn}{\log(1/\eps)}\leq 1+\frac{\log (A_{n-1}/B_n)}{ \log (1/\eps)} - \frac{C}{\log(\sqrt[n]{b_n})}.
 \end{align*}
Since $ \sqrt[n]{b_{n}}\to0$ as $n\to0$ we obtain (\ref{upper-bound}) for $\eps\in[a_n,b_n]$.

Since the right hand side in the first inequality in (\ref{upper-bound}) is decreasing in $\eps$ we obtain that for $\eps\in[b_n,a_{n-1}]$ we have $\mathbb{M}_{I_0,J_0}(\eps) \leq M_n+o(1)$.  Similarly, the right hand side in the second inequality in (\ref{upper-bound}) is increasing in $\eps$ and hence the maximum in the interval $[a_n,b_n]$ is attained for $\eps=b_n$.  This gives (\ref{ineq:limsup-upper-bound}) in $[a_n,b_n]$, thus completing the proof.
 \end{proof}

\subsection{Completing the proof of Lemma \ref{lemma:limsup&liminf}} 
By Proposition \ref{prop:lower-bounds} we have that the left hand side in (\ref{eq:liminf}) is greater than or equal to $m$.  To show the opposite inequality,  let $n_k$ be such that $\lim_{k\to\infty} m_{n_k} = m$. Then letting $\eps_k=a_{n_k}$ in the second inequality in (\ref{upper-bound}) we obtain 
\begin{align}
\mathbb{M}_{I_0,J_0}(\eps_k) \leq 1+\frac{\log {A_{n_k-1}}/{B_{n_k}}}{\log{1}/{a_{n_k}}}+o(1)\underset{k\to\infty}{\longrightarrow} m,
\end{align}
which proves (\ref{eq:liminf}).

Similarly,  Proposition \ref{prop:upper-bounds} implies that $$\limsup_{\eps\to0}\mathbb{M}_{I_0,J_0}(\eps)\leq \limsup_{n\to\infty} (M_n+o(1))=M,$$ and therefore we only need to show the opposite inequality.  Let $n_l$ be such that $\lim_{l\to\infty} M_{n_l} = M$.  Then letting $\eps_l=b_{n_l}$ in the second inequality in (\ref{lower-bound}) we obtain 
\begin{align}
\mathbb{M}_{I_0,J_0}(\eps_{l})&\geq  1+\frac{\log {A_{n_l-1}}/{B_{n_l}}}{\log{1}/{b_{n_l}}}\underset{l\to\infty}{\longrightarrow} M.
\end{align}
This completes the proof of  inequality (\ref{eq:limsup}) and, as explained in the beginning of this section, also proves Theorem \ref{thm:main}.

\section{Limit sets of higher dimensions}\label{section:higher-dim}
%
%

In this section,  for every $k\geq 2$ we construct a domain with chimneys $W_k$ such that the generalized Teichm\"uller geodesic ray $\eps\mapsto \tau_{W_k,\eps}$ has a limit set $\Lambda$ in Thurston boundary $T_{\infty}(\mathbb{D})$ which is a $k$ simplex, i.e., is homeomorphic to $[0,1]^k$.  We will give the detailed construction for $k=2$ and explain how the case of general $k\geq 2$ can be constructed in a similar way.

\subsection{Limit sets of dimension $2$}\label{subsection:2d}

Let $p,q\in(1, \infty)$ and $a\in(0,1)$.  Let $b_0=1$, $a_0=a$, $b_1=a^p$.  For $n\geq 1$,  let $a_n=b_n^p$,  and  $b_n=a_{n-1}^{p}$.  Therefore, for $n\geq 1$ we have
\begin{align*}
a_n= a_{n-1}^{p^2}=a^{p^{2n}},  \quad
b_n=a^{p^{2n-1}}.
\end{align*}

Also, let  $d_0=2$,  $c_0=3-a$,  $d_1=3-a^q$. For $n\geq 1$, let $c_n=3-(3-d_n)^q$, and $d_n = 3-(3-c_{n-1})^q$.  Therefore for $n\geq 1$ we have
\begin{align*}
c_n=3-a^{q^{2n}}, \quad
d_n=3-a^{q^{2n-1}}.
\end{align*} 
Below it will be convenient to use the following notation
\begin{align}
u_n=3-c_n=a^{q^{2n}}, \quad
v_n=3-d_n=a^{q^{2n-1}}.
\end{align}  
Moreover, we let 
\begin{align}
U_n = \prod_{i=0}^n u_i, \quad
V_n =\prod_{i=0}^n v_i.
\end{align}  

Define the chimneys $C_n=(a_n,b_n)\times \mathbb{R}$ and $C_n'=(c_n,d_n)\times  \mathbb{R}$ and the domain $W=W_{p,q}$ by 
\begin{align}\label{domain2}
W_{p,q}=\{z:\Im{z}<0, 0<\Re(z)<3\} \cup \bigcup _{i=1}^{\infty}C_n\cup C_n'. 
\end{align}
We denote by $\psi: \mathbb{D}\to W$ the conformal map such that
\begin{align}
\begin{split}
\lim_{z\to0}\psi^{-1} (z) &=-1,\\
\lim_{z\to3}\psi^{-1} (z) &= 1,\\
\lim_{\Im(z)\to -\infty}\psi^{-1} (z) &= -i.
\end{split}
\end{align}
We denote by $a_n', b'_n,c'_n$ and $d'_n$ the preimages under $\psi$ of $a_n,b_n,c_n$ and $d_n$, respectively.  Also, we define $z_n$ and $z'_n$ as follows:
\begin{align*}
\lim_{\substack{\Im (z) \to +\infty\\ z\in C_n }} \psi^{-1} (z) &= z_n,\\
\lim_{\substack{\Im (z) \to +\infty\\ z\in C'_n}} \psi^{-1} (z) &= z'_n.
\end{align*}
As before we denote by $\g_{(x,y)}$ the hyperbolic geodesic in $\mathbb{D}$ with endpoints $x,y\in\mathbb{S}^1$.  Moreover,  for $n\geq 0$ we denote by $\g_n^{+}$ and  $\g_n^{-}$ the two geodesics starting at $z_n$ and ending at $a'_n$ and $b'_n$, respectively.  Similarly,  for $n\leq 0$ we denote by $\tilde{\g}_n^{+}$ and  $\tilde{\g}_n^{-}$ the geodesics starting at $z'_n$ and ending at $c'_n$ and $d'_n$, respectively.  With this notation in hand we define the measured lamination $\lambda_W$ by
\begin{align}
\lambda_W = \sum_{n=0}^{\infty} (\delta_{\g_i^{+}} + \delta_{\g_i^{-}}) + \sum_{n=0}^{\infty} (\delta_{\tilde{\g}_i^{+}} + \delta_{\tilde{\g}_i^{-}})
\end{align}
%
%
Let $g_1=\g_{(-1,-i)}$ and $g_{2}=\g_{(1,-i)}$.  

Theorem \ref{thm:divergent-square} gives an example of a geodesic ray in $T(\mathbb{D})$ with the limit set in Thurston boundary being homeomorphic to a square.

\begin{theorem}\label{thm:divergent-square}
Suppose $p,q>1$ are such that $\log p/\log q$ is an irrational number.
If the domain $W$ is defined as in (\ref{domain2}) then the generalized Teichm\"uller geodesic $\eps\mapsto\tau_{\eps,W}$ diverges and the accumulation set $\Lambda$ in $\partial_{\infty} T(\mathbb{D})$ can be described as follows:
\begin{align}\label{limitset-2}
\Lambda =   \{ \llbracket \lambda_{s,t} \rrbracket : (s,t)\in [m_p,M_p]\times [m_q,M_q] \},
\end{align}
where  $\lambda_{s,t}= s \delta_{g_1} + t \delta_{g_2} + \frac{2}{3}\lambda_W $, $m_{\ell}=1+\frac{1}{1+\ell}$ and $M_{\ell}=2-\frac{1}{1+\ell}$.
\end{theorem}

Observe that all the measured laminations $\lambda_{s,t}$ are supported on the same geodesic lamination, which we will denote by $\Sigma$. Thus,  $$\Sigma = g_1\cup g_2 \bigcup_{n=0}^{\infty} \g_{n}^+ \cup \g_{n}^- \cup \tilde{\g}_{n}^+ \cup \tilde{\g}_{n}^-.$$

\begin{proof} We first show the inclusion $\supset$ in (\ref{limitset-2}).
For this we will show that for every $(s,t)\in[m_p,M_p]\times[m_q,M_q]$ there is a sequence $\eps_{n_k}\to 0$ such that 
\begin{align}\label{limitset-st}
\tau_{\eps_{n_k},W}\underset{k\to\infty}{\longrightarrow} \llbracket\lambda_{s,t}\rrbracket.
\end{align} 
%

To formulate (\ref{limitset-st}) in terms of the asymptotics of moduli of curve families we let $I_0= \partial(W\cap\{\Im z >0\})$, $J_0 = \{0\}\times(-\infty,-1)$,  $J_1=\{3\} \times (-\infty,-1)$ and denote 
\begin{align*}
\G^{\eps}&=V_{\eps}(\G(I_0,J_0;W))=\G(I_0,\{0\}\times(-\infty,-\eps);W)\\
\tilde{\G}^{\eps}&=V_{\eps}(\G(I_0,J_1;W))=\G(I_0,\{3\}\times(-\infty,-\eps);W).
\end{align*}


Just like in Section \ref{section:independence} to prove (\ref{limitset-st}) it is enough to show that if $\Sigma \cap  I\times J = \emptyset$ then $\lim_{k\to\infty}\mathbb{M}_{I,J}(\eps_{n_k})=0$, while if $\#(\Sigma \cap  I\times J) =1$ then
\begin{align}\label{modlimits-2d}
\lim_{k\to\infty}\mathbb{M}_{I,J}(\eps_{n_k}) = 
\begin{cases}
s, & \mbox{ if } \,\, \Sigma \cap  I\times J = g_1,\\
t, & \mbox{ if } \,\, \Sigma \cap  I\times J = g_2,\\
\frac{2}{3}, & \,\, \mbox{otherwise}.
\end{cases}
\end{align} 

The third case in (\ref{modlimits-2d}) follows just like in \cite{HakSar-visual}.  Moreover,  arguing like in Section \ref{section:independence} we see that instead of general boxes $I\times J$  it is enough to consider the limits of $\mathbb{M}_{I_0,J_0}(\eps_{n_k})$ and  $\mathbb{M}_{I_0,J_1}(\eps_{n_k})$.  Therefore (\ref{limitset-st}) follows from the following lemma.

\begin{lemma}\label{lemma:modlimits-2d}
For every $(s,t)\in[m_p,M_p]\times[m_q,M_q]$ there is a sequence $\eps_{n_k}\to0$ such that
\begin{align}\label{modlimits:st-particular}
\begin{split}
\lim_{k\to\infty}\mathbb{M}_{I_0,J_0}(\eps_{n_k}) &= s,\\
\lim_{k\to\infty}\mathbb{M}_{I_0,J_1}(\eps_{n_k}) &= t. 
\end{split}
\end{align}
\end{lemma}
\begin{proof}[Proof of Lemma \ref{lemma:modlimits-2d}]
From Propositions \ref{prop:lower-bounds} and \ref{prop:upper-bounds} for $n\geq 1$ we have  
\begin{align}\label{modest:I_0xJ_0}
\begin{split}
0&\leq  \mathbb{M}_{I_0,J_0}(\eps)-\left[2-\frac{\log {B_{n}}/{A_{n}}}{\log{1}/{\eps}}\right]\leq o(1),  \quad \mbox{ for } \eps\in[b_{n+1},a_{n}]\\
0&\leq \mathbb{M}_{I_0,J_0}(\eps)-\left[ 1+\frac{\log {A_{n-1}}/{B_{n}}}{\log{1}/{\eps}}\right] \leq o(1),  \quad \mbox{ for } \eps\in[a_n,b_n].
\end{split}
\end{align}

For $\alpha\in[1,p^2]$ we consider $\eps_n=b_n^{\alpha}$.  From the definitions we have  $\eps_n\in[a_n,b_n]$ for $\alpha\in[1,p]$ and $\eps_n\in[b_{n+1},a_n]$ for $\alpha\in[p,p^2]$.  Therefore, 
\begin{align*}
\lim_{n\to\infty}\frac{\m \G^{\eps_n}}{\frac{1}{\pi}\log\frac{1}{\eps_n}}
&=
\begin{cases}
\displaystyle{\lim_{n\to\infty}}1+\frac{\log A_{n-1}/B_{n}}{\alpha\log\frac{1}{b_n}},&    \mbox{ for } \alpha\in[1,p]\\
\displaystyle{\lim_{n\to\infty}}2-\frac{\log {B_{n}}/{A_{n}}}{\alpha\log\frac{1}{b_n}},& \mbox{ for } \alpha\in[p,p^2].
\end{cases}
\end{align*}
From equations (\ref{example:p-a_n}) and (\ref{example:p-b_n})  in Example \ref{example:p}, it follows that $\displaystyle{\lim_{n\to\infty}}\mathbb{M}_{I_0,J_0}(b_n^{\alpha}) = \Phi_p(\alpha)$, where
\begin{align}
\Phi_p(\alpha) = 
\begin{cases}
1+\frac{p}{\alpha(p+1)},  & \mbox{ for } \alpha\in[1,p]\\
2-\frac{p^2}{\alpha(p+1)}, & \mbox{ for } \alpha\in[p,p^2].
\end{cases}
\end{align}
Observe that   $\Phi_p(\alpha)$ decreases from $\Phi_p(1) = M_p$ to $\Phi_p(p) = m_p$ on $\alpha\in[1,p]$ and then increases from $m_p$ to $\Phi_p(p^2)=\Phi_p(1)=M_p$ on $\alpha\in[p,p^2]$. Therefore, for every $s\in[m_p,M_p]$ we can find $\alpha\in[1,p]$ (or  $\alpha\in[p,p^2]$) so that $\Phi_p(\alpha)=s$, and with this choice of $\alpha$ we have that 
$\mathbb{M}_{I_0,J_0}(b_{n}^{\alpha}) \to s$ as $n\to\infty$. In particular, for every subsequence $\eps_{n_k}$ of $\eps_n=b_n^{\alpha}$ the first equation in (\ref{modlimits:st-particular}) holds, but not necessarily the second equation.

Next,  observe that for every $n$ there is a unique $m=m_n\in\mathbb{N}$ such that 
$$\eps_n=b_n^{\alpha} = a^{\alpha p^{2n-1}} \in (v_{m+1},v_{m}]=(a^{q^{2m+1}},a^{q^{2m-1}}].$$
Therefore,  $\exists \,\, \beta_n\in[0,1)$ such that $\eps_n=v_m^{q^{2\beta_n}}=a^{q^{2m-1+2\beta_n}}$, or
$$p^{2n-1+\log_p \alpha} = q^{2m-1+2\beta_n}. $$ Solving for $\beta_n$ we obtain 
\begin{align*}
\beta_n &= \left(2n -1+ \frac{\log \alpha}{\log p}\right) \frac{\log p}{2\log q} +\frac{1}{2} - m= \theta n + \sigma  -m,
\end{align*}
where $\theta: = \log_q p$ and $\sigma := \frac{1}{2}(1+ \log_q \frac{\alpha}{p})$. 
Hence 
\begin{align}
\beta_n \equiv  \theta n + \sigma \,\, (\mathrm{mod} \,1) \equiv T^{\circ n}_{\theta}(\sigma)
\end{align}
is the $n$-th iterate of $\sigma$ under the map $T_{\theta}:x\mapsto x+\theta \,\, ( \mathrm{mod} \, 1)$. Since $\theta\in\mathbb{R} \setminus \mathbb{Q}$ the map $T_{\theta}$ is an irrational rotation,  and every orbit $\{T_{\theta}^{\circ n}(\sigma)\}_{n\in\mathbb{N}}$ is dense in $[0,1]$, see for instance \cite{BrinStuck}. Therefore for every $\beta\in[0,1]$ there is a sequence $n_k$ such that $\beta_{n_k} \underset{k\to\infty}{\longrightarrow} \beta$. Moreover, $\beta_{n_k}$ can be chosen to be a monotone sequence.
%
We claim that 
\begin{align}\label{modlimit:weight}
\lim_{k\to\infty}\mathbb{M}_{I_0,J_1}(\eps_{n_k})=\Phi_q(q^{2\beta}).
\end{align}


As in the proofs of Lemma \ref{lemma:limsup&liminf}, Propositions \ref{prop:lower-bounds} and \ref{prop:upper-bounds}, we have the following estimates:
\begin{align}\label{modest:I_0xJ_1}
\begin{split}
0&\leq  \mathbb{M}_{I_0,J_1}(\eps)-\left[2-\frac{\log {V_{n}}/{U_{n}}}{\log{1}/{\eps}}\right]\leq o(1),  \quad \mbox{ for } \eps\in[v_{n+1},u_n]\\
0&\leq \mathbb{M}_{I_0,J_1}(\eps)-\left[ 1+\frac{\log {U_{n-1}}/{V_{n}}}{\log{1}/{\eps}}\right]\leq o(1),  \quad \mbox{ for } \eps\in[u_n, v_n].
\end{split}
\end{align}
Therefore,  since $\eps_{n_k}=(a^{q^{2m-1}})^{q^{2\beta_{n_k}}}=(v_m)^{q^{2\beta_{n_k}}}$ we have
\begin{align}\label{modest:I_0xJ_1-2}
\begin{split}
0&\leq  \mathbb{M}_{I_0,J_1}(\eps_{n_k})- \left[2-\frac{\log {V_{m}}/{U_{m}}}{q^{2\beta_{n_k}}\log{1}/{v_m}}\right]\leq o(1),  \quad \mbox{ for } \eps_{n_k}\in[v_{m}^{q^2},v_m^q]\\
0&\leq \mathbb{M}_{I_0,J_1}(\eps_{n_k}) -\left[ 1+\frac{\log {U_{m-1}}/{V_{m}}}{q^{2\beta_{n_k}}\log{1}/{v_m}}\right]\leq o(1),  \quad \mbox{ for } \eps_{n_k}\in[v_m^q, v_m].
\end{split}
\end{align}
Note that $\eps_{n_k}=(v_m)^{q^{2\beta_{n_k}}}\in[v_{m}^{q},v_m]$  if and only if $\beta_{n_k}\in[0,1/2]$.  Therefore,  for any $\beta\in[0,1/2]$ choosing $\beta_{n_k}\in[0,1/2]$ so that $\beta_{n_k}\to \beta$ we obtain
\begin{align}\label{modlimit:density}
\begin{split}
\lim_{k\to\infty}\mathbb{M}_{I_0,J_1}(\eps_{n_k}) 
&=\lim_{k\to\infty} 1+\frac{\log {U_{m-1}}/{V_{m}}}{q^{2\beta_{n_k}}\log{1}/{v_m}}\\
&=1+\frac{1}{q^{2\beta}} \lim_{m\to\infty}\frac{(1/q-1)\log V_m}{\log 1/v_m}.\\
&=1+\frac{1}{q^{2\beta}}\left(\frac{1}{q}-1\right) \lim_{m\to\infty}\frac{q^{2m+2}-1}{(q^2-1)q} \cdot \frac{1}{-q^{2m-1}}\\
&=1+\frac{1}{q^{2\beta}}\frac{q}{q+1}=\Phi_q(q^{2\beta}) .\
\end{split}
\end{align}

Similarly,   $\eps_{n_k}=(v_m)^{q^{2\beta_{n_k}}}\in[v_{m}^{q^2},v_m^q]$ if and only if $\beta_{n_k}\in[1/2,1]$.   Therefore,  for any $\beta\in[1/2,1]$ choosing $\beta_{n_k}\in[1/2,1]$ so that $\beta_{n_k}\to \beta$ we obtain
\begin{align*}
\lim_{k\to\infty}\mathbb{M}_{I_0,J_1}(\eps_{n_k}) 
&=\lim_{k\to\infty} 2-\frac{\log {V_{m}}/{U_{m}}}{q^{2\beta_{n_k}}\log{1}/{v_m}}
= 2-\frac{1}{q^{2\beta}}\frac{q^2}{q+1}=\Phi_q(q^{2\beta}),
\end{align*}
which proves (\ref{modlimit:weight}).

Note that  $\Phi_q(q^{2\beta})$ continously decreases from $\Phi_q(1)=M_q=1+\frac{q}{q+1}$ to $\Phi_q(q)=m_q=1+\frac{1}{q+1}$ as $\beta\in[0,1/2]$.  Similarly,  $\Phi_q(q^{2\beta})$ increases from $m_q$ to $M_q$ in $\beta\in[1/2,1]$. Therefore, for every $t\in[m_q,M_q]$ there is a $\beta\in[0,1/2]$ (or $\beta\in[1/2,1]$) such that $\Phi_q(q^{2\beta})=t$.  For this choice of $\beta$ choosing 
$n_k$ so that (\ref{modlimit:weight}) holds gives the second line in (\ref{modlimits:st-particular}), thus proving Lemma \ref{lemma:modlimits-2d}.
\end{proof}
To complete the proof  of Theorem \ref{thm:divergent-square} we need to show that if $\lambda$ is a geodesic measured lamination such that $\tau_{\eps_n,W}\underset{n\to\infty}\longrightarrow\llbracket\lambda\rrbracket$ for some sequence $\eps_n\to0$ then $\lambda=\lambda_{s,t}$ with $(s,t)\in[m_p,M_p]\times[m_q,M_q]$. Recall that just like in \cite{HakSar-visual} we have that  $\mathbb{M}_{I,J}(\eps)\to2/3$ as $\eps\to0$ whenever $\Sigma \cap (I\times J) = \g$ for some $\g\in \Sigma\setminus{(g_1\cup g_2)}$.  Hence it follows from Corollary \ref{corol:modformulation}, Lemma \ref{lemma:discrete-laminations} and Lemma \ref{lemma:modlimits-2d} that to prove Theorem \ref{thm:divergent-square} it is enough to show that the following inequalities hold:
\begin{align*}
\limsup_{\eps\to0}\mathbb{M}_{I_0,J_i}(\eps)\leq
\begin{cases}
M_p, & \mbox{ if } i=0,\\
M_q, & \mbox{ if } i=1,
\end{cases}\\
\liminf_{\eps\to0}\mathbb{M}_{I_0,J_i}(\eps)\geq
\begin{cases}
m_p, & \mbox{ if } i=0,\\
m_q, & \mbox{ if } i=1.
\end{cases}
\end{align*}
However, these follow from inequalities (\ref{modest:I_0xJ_0}) and (\ref{modest:I_0xJ_1}) applied to the Example \ref{example:p}.
\end{proof}

\subsection{Limit sets of dimension $k\in[2,\infty]$} Let $2\leq k <\infty$. For all $p,q\in(1,\infty)$ let $W_{p,q}\subset\{0<\Re(z)<3\}$ denote the domain constructed in Section \ref{subsection:2d} above.  Given a domain $W\subset \mathbb{C}$ and $z\in\mathbb{C}$ we denote  $W+z=\{w+z | w\in W\}$, i.e.  the Minkowski sum of $W$ and $z$, or equivalently $z$ translate of $W$.

Suppose  $p_j\in(1,\infty)$ for $j\in\{1,\ldots,k\}$ and $a\in(0,1)$.  For $n\geq 1$,  let 
\begin{align*}
a_{j,n} = a^{p_j^{2n}}, \quad b_{j,n}=a_{j,n}^{1/p} = a^{p_j^{2n-1}},  \quad c_{j,n}=3-a^{p_j^{2n}}, \quad
d_{j,n}=3-a^{p_j^{2n-1}}.
\end{align*}
Also, let  $C_{j,n}=(a_{j,n},b_{j,n})\times (-\infty,\infty)$ and $C'_{j,n}=(d_{j,n},c_{j,n})\times (-\infty,\infty)$.  We define the domains $W_j:=W_{p_{j}, p_{j}} +3j$, where
\begin{align}
W_{p_{j}, p_{j}}=\bigcup_{n=0}^{\infty} (C_{j,n}\cup C_{j,n}') \cap \{ z\, | \,0<\Re(z)<3,  \Im(z)<0\}
\end{align}

\begin{figure}\label{figure:k-dim}
	\centerline{
		\includegraphics[width=1\textwidth]{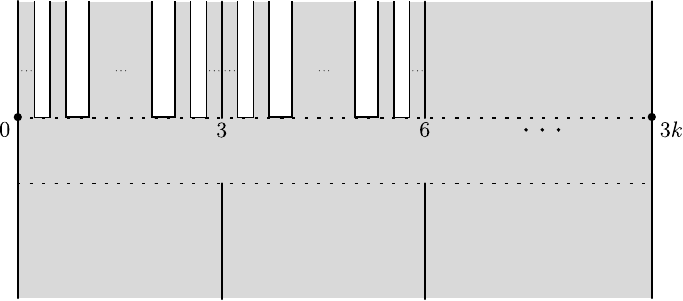}
	}
	\caption{\small{Construction of domain $\mathcal{W}_k$}.}
\end{figure}

Note that the domains $W_j$ are pairwise disjoint, but $\overline{W_{j}} \cap \overline{W_{j+1}} = \{ z\in\mathbb{C} \,  | \, \Re z =3j\}$.  Define the domain $\mathcal{W}_k$ as follows:
\begin{align}\label{def:domain_k}
\mathcal{W}_k = \bigcup_{j=1}^k W_j \cup ((0,3k)\times(-1,0)).
\end{align}
Thus $\mathcal{W}_k$ is obtained from $\cup_{j=1}^k W_j$ by ``adding'' the vertical intervals $\{3j\}\times(-1,0) $ which can be thought of as ``channels'' connecting $W_j$ and $W_{j+1}$ in $\mathcal{W}_k$.  

Let $\psi_k:\mathbb{D}\to \mathcal{W}_k$ be a conformal map such that 
\begin{align}
\begin{split}
\lim_{z\to0}\psi_k^{-1} (z) &=-1,\\
\lim_{z\to3k}\psi_k^{-1} (z) &= 1,\\
\lim_{\Im(z)\to -\infty}\psi_k^{-1} (z) &= -i.
\end{split}
\end{align}
We define the geodesics $\g_{n,j}^{\pm}$ and $\tilde{\g}_{n,j}^{\pm}$ for $j\in\{1,\ldots,k\}$ and $n\geq0$ analogously to Section \ref{subsection:2d}, i.e., for every chimney in $\mathcal{W}_k\cap W_j$ we have two corresponding geodesics in the unit disk with a common endpoint. We let 
\begin{align}
\lambda_{\mathcal{W}_k} = \sum_{j=1}^k \left(\sum_{n=0}^{\infty} (\delta_{\g_{n,j}^{+}} + \delta_{\g_{n,j}^{-}}) + \sum_{n=0}^{\infty} (\delta_{\tilde{\g}_{n,j}^{+}} + \delta_{\tilde{\g}_{n,j}^{-}})\right).
\end{align}
Letting $\xi_j:= \lim_{w\to3j} \psi_k^{-1}(w)$ for $j\in\{0,\ldots,k\}$, and 
\begin{align*}
\zeta_j &= \lim_{\substack{\Im (z) \to \infty\\ z\in W_j }} \psi_k^{-1}(w), \quad 
\end{align*}
for $j\in\{1,\ldots, k\}$ we also denote by $g_j$ and $g_j'$ the hyperbolic geodesics in $\mathbb{D}$ connecting $\zeta_j$ to $\xi_{j-1}$ and $\xi_j$, respectively.


Recall that real numbers $\theta_1,\ldots,\theta_k$ are \emph{rationally independent} if the equation $r_1\theta_1+\ldots+r_k\theta_k=0$ with integer coefficients $r_j$ holds only if $r_1=\ldots=r_k=0$.

\begin{theorem}\label{thm:kd}
Suppose the collection of numbers $\log p_1,\ldots, \log p_k$ is rationally independent. 
If the domain $\mathcal{W}_k$ is defined as in (\ref{def:domain_k}) then the generalized Teichm\"uller geodesic $\eps\mapsto\tau_{\eps,\mathcal{W}_k}$ diverges and the accumulation set $\Lambda$ in $\partial_{\infty} T(\mathbb{D})$ can be described as follows:
\begin{align}\label{limitset-2}
\Lambda =   \left\{ \llbracket \lambda_{s_1,\ldots,s_k} \rrbracket : (s_1,\ldots,s_k)\in \prod_{j=1}^k [m_{p_j},M_{p_j}]\right\},
\end{align}
where  
\begin{align}
\lambda_{s_1,\ldots,s_k}= \sum_{j=1}^k s_j(\delta_{g_j} + \delta_{g_j'}) + \frac{2}{3}\lambda_{\mathcal{W}_k}.
\end{align}
\end{theorem}
The proof of Theorem \ref{thm:kd} is a generalization of the proof of Theorem \ref{thm:divergent-square}.  The key difference is that we use the classical Kronecker approximation theorem about the density of the sequence $(\{\theta_1n\},\ldots,\{\theta_k n\})$ in the $k$-dimensional unit cube when $\theta_1,\ldots,\theta_k$ are rationally independent, instead of the the fact that every orbit $\{T^{\circ n}_{\theta}(x)\}$ of an irrational rotation is dense in $\mathbb{S}^1$.
For this reason we provide a detailed sketch of the proof Theorem \ref{thm:kd}, and leave the verification of the few missing details to the reader.  

\begin{proof}[Sketch of the proof of Theorem \ref{thm:kd}] The inclusion $\subset$ in (\ref{limitset-2}) follows from the inequalities $m_j\leq\liminf_{\eps\to0}M_{I_j,J_j}(\eps)$ and $\liminf_{\eps\to0}M_{I_j,J_j}(\eps)\leq M_j$ just like in Theorem \ref{thm:divergent-square}.

Let $\Sigma_k$ denote the support of a measured lamination $\lambda_{s_1,\ldots,s_k}$, which is independent of the choice of the weights $s_j$.   The proof of the inclusion $\supset$  in  (\ref{limitset-2}),  just like  for Theorem \ref{thm:divergent-square},  comes down to showing that for every choice of  $(s_1,\ldots,s_k)\in\prod_{j=1}^k [m_{p_j},M_{p_j}]$ there exists a sequence $\eps_{n_k}$ so that if the only geodesics from $\Sigma_k$ contained in the box $I\times J$ are either $g_j$ or $g_j'$ for some $j\in\{1,\ldots,k\}$ then 
\begin{align}\label{modlimit:kd}
\lim_{n\to\infty} \mathbb{M}_{I,J}(\eps_{n_k}) = s_j.
\end{align}
In the cases when  $\Sigma_k \cap (I\times J)=g\notin \{g_j, g_j'\}$ and $\Sigma_k \cap (I\times J)=\emptyset$.  In these cases we have that $\lim_{n\to\infty} \mathbb{M}_{I,J}(\eps_n)$  is  equal to  either $2/3$ or $0$, respectively,  which is proved just like in Theorem  \ref{thm:divergent-square} or rather in \cite{HakSar-visual}.

Let $I_j=\partial \mathcal{W}_k \cap \{3j< \Re z < 3j+1\} \cap \{\Im z>0\}$, for $j\in\{0,\ldots, k-1\}$,  $J_j=\{3j\} \times (-\infty,-1), $ for $j\in\{0,\ldots, k\}$. Then, like in the proof of Theorem      \ref{thm:divergent-square} it is enough to show that for $j\in\{0,\ldots,k-1\}$ we have
\begin{align}\label{modlim:s_j}
\lim_{n\to\infty} \mathbb{M}_{I_j,J_j}(\eps_n) = \lim_{n\to\infty} \mathbb{M}_{I_j,J_{j+1}}(\eps_n)= s_j.
\end{align}

As in the proof of Theorem \ref{thm:divergent-square} we choose $\alpha_1$ so that $\Phi_{p_1}(\alpha_1)=s_1$, let $\eps_n=b_{1,n}^{\alpha_1}$, and observe that for every $j\in\{1,\ldots,k\}$ and $n\geq 1$ there is a unique $m_j=m_{j,n}\in\mathbb{N}$ such that 
$\eps_{n}=b_{1,n}^{\alpha_1} = a^{\alpha_1 p_j^{2n-1}} \in (a^{p_j^{2m_j+1}},a^{p_j^{2m_j-1}}].$
Therefore,  $\exists \,\, \beta_{j,n}\in[0,1)$ such that $\eps_{n}=a^{q^{2m_j-1+2\beta_{j,n}}}$, or
$p_1^{2n-1+\log_{p_1} \alpha_1} = q^{2m_j-1+2\beta_{j,n}}. $ Solving for $\beta_{j,n}$ we obtain 
\begin{align*}
\beta_{j,n} &= \left(2n -1+ \frac{\log \alpha}{\log p_1}\right) \frac{\log p_1}{2\log p_j} +\frac{1}{2} - m_j= \theta_j n + \sigma_j  -m_j,
\end{align*}
where $\theta_j: = \log_{p_j} p_1$ and $\sigma_j:= \frac{1}{2}(1+ \log_{p_j} \frac{\alpha_1}{p_1})$.  Hence,  for $j\in\{1,\ldots,k\}$ we have
\begin{align}
\beta_{j,n} \equiv \theta_j n + \sigma_j   (\mathrm{mod} 1).
\end{align}


Next we  use the classical  theorem of Kronecker, see \cite[Thm.  443, p.  382]{HardyWright}.

\begin{theorem}[Kronecker's density theorem]
If $\theta_1, \theta_2,\ldots,\theta_k$ are rationally independent,  $x_1,\ldots,x_k$ are arbitrary, and $N$ and $\eps$ are positive, then there are integers $n>N$, and $p_1,\ldots,p_k$ such that 
\begin{align}
|n\theta_j - p_j-x_j| <\eps,
\end{align}
for all $j\in\{1,\ldots,k\}$.
\end{theorem}

Therefore,  for every choice of $\beta_1,\ldots,\beta_k$ there are sequences of natural numbers $n_k$ and $p_{j,k}$ such that 
$$|n_k\theta_j - p_{j,k}-\sigma_j-\beta_j| <1/k$$ for all $j\in\{1,\ldots,k\}$.  Hence, for this choice of $n_k$ we have that $\beta_{j,n_k} \to \beta_j$ as $k\to\infty$ for all $j\in\{1,\ldots,k\}$.  Moreover,  arguing  as in the proof of Lemma \ref{lemma:modlimits-2d} (the part after equation (\ref{modest:I_0xJ_1})), we can see that for an arbitrrary choice of $s_j$ we can choose the sequence $n_k$ so that for every $j\in\{1,\ldots,k\}$ we have
\begin{align*}
\lim_{k\to\infty}\mathbb{M}_{I_j,J_j}(\eps_{n_k})=\lim_{k\to\infty}\Phi_{p_j}(p_j^{2\beta_{j,n_k}})= \Phi_{p_j}(p_j^{2\beta_j})=s_j.
\end{align*}
This proves the inclusion $\supset$ in (\ref{limitset-2}) and completes the proof of the theorem.
\end{proof}
%
%
%

\end{document}